\documentclass[11pt]{article}

\usepackage{verbatim,latexsym,amsfonts,amsmath,amssymb,graphicx,fancyhdr,hyperref,asymptote,enumitem}
\usepackage{appendix,latexsym,amsfonts,amsmath,amssymb,graphicx,hyperref,amsthm,soul,verbatim,authblk}
\usepackage[framemethod=tikz]{mdframed}

\newcommand{\EE}{\mathbb{E}}
\newcommand{\PP}{\mathbb{P}}

\newcommand{\R}{\mathbb{R}}
\newcommand{\C}{\mathbb{C}}

\newcommand{\HH}{\mathbb{H}}
\newcommand{\N}{\mathbb{N}}
\newcommand{\D}{\mathbb{D}}
\newcommand{\Z}{\mathbb{Z}}

\newcommand{\TT}{\mathbb{T}}

\newcommand{\CC}{{\cal C}}

\newcommand{\pa}{\partial}

\newcommand{\K}{{\cal K}}

\newcommand{\F}{{\cal F}}

\newcommand{\no}{\noindent}

\newcommand{\Lo}{{\cal L}}

\def\eps{\varepsilon}
\def\til{\widetilde}
\def\ha{\widehat}
\def\sem{\setminus}
\def\lin{\overline}

\def\up{\upsilon}
\def\Up{\Upsilon}

\def\tl{\vartriangleleft}

\DeclareMathOperator{\sign}{sign} 
\DeclareMathOperator{\dist}{dist} 
 \DeclareMathOperator{\id}{id}
\DeclareMathOperator{\Imm}{Im } \DeclareMathOperator{\Ree}{Re }

\DeclareMathOperator{\mA}{m}

\DeclareMathOperator{\arcsinh}{arcsinh}
\newtheorem{Lemma}{Lemma}[section]
\newtheorem{Theorem}{Theorem}[section]
\newtheorem{Definition}{Definition}[section]
\newtheorem{Corollary}{Corollary}[section]
\newtheorem{Proposition}{Proposition}[section]
\numberwithin{equation}{section}
\newcommand{\BGE}{\begin{equation}}
\newcommand{\BGEN}{\begin{equation*}}
\newcommand{\EDE}{\end{equation}}
\newcommand{\EDEN}{\end{equation*}}

\begin{document}
\title{Decomposition of  Schramm-Loewner evolution along its curve}
\author{Dapeng Zhan\footnote{Research partially supported by NSF grants DMS-1056840 and Sloan fellowship}}
\affil{Michigan State University}
\maketitle
\abstract{We show that, for $\kappa\in(0,8)$, the integral of the laws of two-sided radial SLE$_\kappa$ curves through different interior points against a measure with SLE$_\kappa$ Green's function density is the law of a chordal SLE$_\kappa$ curve, biased by the path's natural length. We also show that, for $\kappa>0$, the integral of the laws of extended SLE$_\kappa(-8)$ curves through different interior points against a measure with a closed formula density restricted in a bounded set is the law of a chordal SLE$_\kappa$ curve, biased by the path's capacity length restricted in that set. Another result is that, for $\kappa\in(4,8)$, if one integrates the laws of two-sided chordal SLE$_\kappa$ curves through different force points on $\R$ against a measure with density on $\R$, then one also gets a law that is absolutely continuous w.r.t.\ that of a chordal SLE$_\kappa$ curve. To obtain these results, we develop a framework to study stochastic processes with random lifetime, and improve the traditional Girsanov's Theorem.\\
\\
\no Keywords: SLE, Girsanov's Theorem, Doob-Meyer decomposition
 }

\section{Introduction}
The Schramm-Loewner evolution (SLE), first introduced by Oded Schramm in 1999 (\cite{S-SLE}), is a one-parameter ($\kappa\in(0,\infty)$) family of measures on non-self-crossing curves, which has received a lot of attention over the past eighteen years. 
It has been shown that, modulo time parametrization,  several discrete random paths on grids (e.g., loop-erased random walk \cite{LSW}, critical percolation explorer \cite{Smir1,percolation-2}) have SLE as a scaling limit.

SLE is defined using Loewner's differential equation, and is originally parameterized by capacity. For the discrete random paths which converge to SLE, in order to show the convergence, people have to first reparameterize them by capacity and then prove that the reparametrized curve converge to SLE with capacity parametrization. The convergence does not take into consideration the discrete length of the path.

In order to upgrade the convergence results, Lawler and Sheffield introduced the natural  parametrization of SLE in \cite{LS}, and conjectured that those discrete random paths with their original length suitably rescaled, converge to the SLE with natural parametrization. Their construction used the Doob-Meyer decomposition, and they proved the existence of the natural  parametrization of SLE$_\kappa$ for  $\kappa <5.021...$. This result was later improved by \cite{LZ}, where it was shown that the natural parametrization of SLE$_\kappa$ exists for all $\kappa\in(0,8)$.

It was proved later in \cite{LR} that the natural parametrization agrees with the $d$-dimensional Minkowski content of the SLE$_\kappa$ curve, where $d=1+\frac\kappa 8$ is the Hausdorff dimension of the curve (cf.\ \cite{dim-SLE}). It was proved recently (cf.\ \cite{LERW-NP2,LERW-NP3,LERW-NP4}) that the loop-erased random walk with natural length converges to SLE$_2$ with  natural parametrization. The work uses an earlier result on the convergence of the density of the loop-erased random walk to the Green's function for SLE$_2$ (cf. \cite{LERW-NP}).

There are two major versions of SLE: chordal SLE and radial SLE. Most of the study focuses on chordal SLE, which describes a curve in a simply connected domain from one prime end (cf.\ \cite{Ahl}) to another prime end. Two-sided radial SLE$_\kappa$ and SLE$_\kappa$ Green's function for $\kappa\in(0,8)$ were introduced in \cite{LS}. A two-sided radial SLE$_\kappa$ curve has two arms: the first arm from a prime end to an interior point is a chordal SLE$_\kappa(\kappa-8)$ process, and the second arm from the interior point to another prime end is a chordal SLE$_\kappa$ curve conditioned on the first arm. It can be understood as a chordal SLE$_\kappa$ curve conditioned to pass through a fixed interior point. While that event has probability zero, some limiting procedure was used to make this idea rigorous. SLE$_\kappa$ Green's function was defined by a closed formula, and turned out to be the density of the SLE$_\kappa$ curve with natural parametrization. These two objects have discrete analogues. The two-sided radial SLE$_\kappa$ curve corresponds to the discrete random path conditioned to pass through a fixed vertex, and the SLE$_\kappa$ Green's function corresponds to the density of the path.

Field recently proved in \cite{Fie} that, for $\kappa\in(0,4]$,  if one integrates the laws of two-sided radial SLE$_\kappa$ in a bounded analytic domain $D$ passing through different interior points (with the two ends fixed) against the measure with density (w.r.t.\ the Legesgue measure on $\C$) being the SLE$_\kappa$ Green's function in $D$, then one gets the law of a chordal SLE$_\kappa$ curve biased by the curve's length in the {\it natural parametrization}. This is analogous  to a simple fact of discrete random paths: if one integrates the laws of the path conditioned to pass through different fixed vertices against the probability that the path passes through each fixed vertex, one should get a measure on paths, which is absolutely continuous w.r.t.\ the law of the original discrete random path, and the Radon-Nikodym derivative is the total number of vertices on the path, which is due to the repetition of counting. 

This paper is motivated by Field's work. We extend his result from $\kappa\in(0,4]$ to $\kappa\in(0,8)$ (see Corollary \ref{Cor-general}) using a new approach. We do not need to assume that the domain is bounded or analytic. The main tools used here are from Probability Theory. 

Our approach is different from Field's  in that we develop a theory of stochastic processes with random lifetime, which is purely probability. The computation here is simpler because the major work: the Doob-Meyer decomposition for the  definition of the natural parametrization was done earlier in \cite{LS,LZ}. On the other hand,
Field's proof  used some SLE technique such as an escape estimate of SLE derived in his joint work \cite{Fie-Law} with Lawler. That results in the technical assumption in his paper that $\kappa\le 4$ and that the domain is bounded with analytic boundary. Another difference is that we study the measures on the space of curve-point pairs instead of just on curves. This makes the main theorem more convenient to be applied later in \cite{loop} to construct SLE loop measures.

After the main theorem, we find other applications of the new technique. We define the extended chordal SLE$_\kappa(\rho)$ curve in the upper half plane $\HH$ for $\rho\le \frac\kappa 2-4$, which is composed of two arms: the first arm growing from $0$ to $z_0\in\HH$ is a chordal SLE$_\kappa(\rho)$ curve, and the second arm growing from $z_0$ to $\infty$ is a chordal SLE$_\kappa$ curve conditioned on the first arm. In particular, a two-sided radial SLE$_\kappa$ curve is an extended chordal SLE$_\kappa(\kappa-8)$ curve. We prove that, for any $\kappa>0$, there is a positive function $G^{\kappa,-8}(z)$ with closed formula such that for any measurable $U\subset\HH$ on which $G^{\kappa,-8}$ is integrable, if one integrates the laws of extended chordal SLE$_\kappa(-8)$ curve through different $z$ against a measure with density ${\bf 1}_U G^{\kappa,-8}(z)$, then one gets the law of a chordal SLE$_\kappa$ curve biased by $C_{\kappa,1}$ times the total time that the curve spends in $U$ in the {\it capacity parametrization} (see Corollary \ref{Cor3}), where $C_{\kappa,1}$ is a positive constant depending only on $\kappa$. 

The above two main results of this paper immediately imply that, if we sample a point on a chordal SLE$_\kappa$ curve in $\HH$ according to a law, which is absolutely continuous w.r.t.\ the natural parametrization (resp.\ capacity parametrization), and stop the curve at that point, then we get a random curve, whose law is absolutely continuous w.r.t.\ that of a chordal SLE$_\kappa(\kappa-8)$ (resp.\ SLE$_\kappa(-8)$) curve in $\HH$.

The connection between SLE$_\kappa$ near its tip at a fixed capacity time and SLE$_\kappa(-8)$ was derived earlier in \cite[Proposition 3.10]{mating}. The tip behavior of SLE was also studied in \cite{tip} for $\kappa\in(0,4)$. What new here is that we derive the Green's function for the SLE in capacity parametrization up to a multiplicative constant. 


Another application of the new technique is to study the intersection of SLE curve with the boundary. Using a Doob-Meyer decomposition, Alberts and Sheffield constructed in \cite{AlbertsSheffield} a measure supported by the intersection of an SLE$_\kappa$ curve $\gamma$ in $\HH$ for $\kappa\in(4,8)$ with $\R$, and conjectured that the measure agrees with the $(2-\frac8\kappa)$-dimensional Minkowski content of $\gamma\cap\R$. Their work used two-sided chordal SLE$_\kappa$, which can be understood as a chordal SLE$_\kappa$ curve conditioned to pass through a fixed point on the boundary. Using their result, we prove in this paper that, if one integrates the laws of two-sided chordal SLE$_\kappa$ curve in $\HH$ from $0$ to $\infty$ through different $x\in\R$ against a measure with a suitable density function (known as boundary Green's function for SLE) against the Lebesgue measure on a bounded interval $I\subset\R$, then one gets a measure on curves, which is absolutely continuous w.r.t.\ the law of a chordal SLE$_\kappa$ curve, and the Radon-Nikodym derivative is the measure constructed in \cite{AlbertsSheffield} of the intersection of the curve with $I$.

The new technique has applications beyond the SLE area. For example, we use it to decompose a planar Brownian motion restricted in a simply connected domain.

In the author's recent preprint \cite{loop}, the decomposition of chordal SLE into two-sided radial SLE is used to construct SLE loop measures. 

The paper is organized as follows. In Section \ref{section-lifetime}, in order to study the driving functions of  SLE$_\kappa(\rho)$ curves, we develop a framework on stochastic processes with random lifetime. We introduce the ``local absolute continuity'' between these processes, and extend the traditional Girsanov's theorem. In Section \ref{section-SLE}, we review the definitions and basic properties of SLE$_\kappa$  and its variants,  SLE$_\kappa(\rho)$ processes.  In Section \ref{section-NP}, we prove our main result about natural parametrization and two-sided radial SLE curves. In Section \ref{section-capacity}, we prove the result about capacity parametrization and extended chordal SLE$_\kappa(-8)$ curves. In Section \ref{section-boundary}, we show the application on boundary measure and two-sided chordal SLE curves. In Section \ref{Section-BM}, we use the technique to decompose planar Brownian motions. In the appendix, we prove the transience property of chordal SLE$_\kappa(\rho)$ curves.

\section*{Acknowledgement} The author acknowledges the support from the National Science Foundation under the grant DMS-1056840 and the support from the Alfred P.\ Sloan Foundation.

\section{Stochastic Processes with Random Lifetime} \label{section-lifetime}
For $0<T\le \infty$, let $C([0,T))$ denote the space of real valued continuous functions on $[0,T)$. Let
$$\Sigma=\bigcup_{0<T\le \infty} C([0,T)).$$
For each $f\in\Sigma$, let $T_f$ be such that $[0,T_f)$ is the domain of $f$.

We define two basic operations on $\Sigma$: killing and continuing.
For $0<\tau\le\infty$, we define the killing map $\K_\tau:\Sigma\to\Sigma$ such that if $g=\K_\tau(f)$, then $T_g=\tau\wedge T_f$ and $g=f|_{[0,T_g)}$.
Let $\Sigma^\oplus=\{f\in\Sigma:T_f<\infty,f(T_f^-):=\lim_{t\to T_f^-} f(t)\in\R\}$ and $\Sigma_\oplus=\{f\in\Sigma:f(0)=0\}$. For example, if $T_f>\tau>0$, then $\K_\tau(f)\in\Sigma^\oplus$. For $f\in\Sigma^\oplus$ and $g\in\Sigma_\oplus$, we may continue $f$ using $g$ and get the function $f\oplus g\in\Sigma$, which is defined by $T_{f\oplus g}=T_f+T_g$ and
$$f\oplus g(t)=\left\{\begin{array}{ll} f(t), & 0\le t<T_f;\\
f(T_f^-)+g(t-T_f), & T_f\le t<T_f+T_g.
\end{array}
\right.$$
Sometimes we want to record the time that the two functions are joined together. For this purpose, we define $f\ha\oplus g=(f\oplus g,T_f)$. Then we can use $f\ha\oplus g$ to recover $f$ and $g$.

For $0\le t< \infty$, let $\F_t^0$ be the $\sigma$-algebra generated by $$\{f\in\Sigma: s<T_f, f(s)\in U\},\quad 0\le s\le t, U\in\cal{B}(\R).$$
Then  $(\F_t^0)$ is a filtration on $\Sigma$. Let $\F^0=\sigma(\bigcup_{0\le t<\infty}\F_t^0)$. We will mainly work on the measurable space $(\Omega,\F^0)$ or its completion w.r.t.\ a certain measure. Every probability measure on $(\Sigma,\F^0)$ is the law of a continuous stochastic process with random lifetime. For any measure $\mu$ on $(\Sigma,\F^0)$, we use $\F^\mu$ and $\F^\mu_t$ to denote the $\mu$-completion of $\F^0$ and $\F^0_t$, respectively.

The continuing maps $(f,g)\mapsto f\oplus g$ and $(f,g)\mapsto f\ha\oplus g$ are measurable. If $\mu$ and $\nu$ are $\sigma$-finite measures supported by $\Sigma^\oplus$ and $\Sigma_\oplus$, respectively, we use $\mu\oplus\nu$ and $\mu\ha\oplus \nu$ to denote the pushforward measures of the product measure $\mu\otimes\nu$ under the maps $(f,g)\mapsto f\oplus g$ and $(f,g)\mapsto f\ha\oplus g$, respectively.

Let's recall an important notion in Probability: kernel. Suppose $(U,{\cal U})$ and $(V,{\cal V})$ are two measurable spaces. A kernel from $(U,{\cal U})$ to $(V,\cal V)$ is a map $\nu:U\times {\cal V}\to[0,\infty]$ such that (i) for every $u\in U$, $\nu(u,\cdot)$ is a measure on $(V,\cal V)$, and (ii) for every $F\in\cal V$, $\nu(\cdot,F)$ is $\cal U$-measurable.
Let $\mu$ be a $\sigma$-finite measure on $(U,\cal U)$. Let ${\cal U}^\mu$ be the $\mu$-completion of $\cal U$. A $\mu$-kernel from $(U,{\cal U})$ to $(V,\cal V)$ is a kernel from $(U^\mu,{\cal U}^\mu\cap U^\mu)$ to $(V,\cal V)$, where $ U^\mu\subset U$ is such that $U\sem U^\mu$ is a $\mu$-null. The $\mu$-kernel is said to be finite if for $\mu$-a.s.\ every $u\in U$, $\nu(u,V)<\infty$; and is said to be $\sigma$-finite if there is a sequence $F_n\in\cal V$, $n\in\N$, with $V=\bigcup F_n$ such that for any $n\in\N$, and $\mu$-a.s.\ every $u\in U$, $\nu(u,F_n)<\infty$. If $\nu$ is a $\sigma$-finite $\mu$-kernel from $(U,{\cal U})$ to $(V,\cal V)$, then we may define a measure $\mu\otimes \nu$ on ${\cal U}\times{\cal V}$ such that
$$\mu\otimes \nu(E\times F)=\int_E \nu(u,F)d\mu(u),\quad E\in\cal U,\quad F\in \cal V.$$
This new measure is first defined on the semi-ring $\{E\times F:E\in\cal U,F\in\cal V\}$, and then extended to a measure on ${\cal U}\times{\cal V}$. Carath\'eodory's extension theorem guarantees the existence of the extension. The $\sigma$-finiteness of $\mu$ and $\nu$ ensures that the extension is unique, and $\mu\otimes\nu$ is also $\sigma$-finite. We use $\mu\cdot\nu$ to denote the marginal of $\mu\otimes\nu$ on $(V,\cal V)$, i.e., $\mu\cdot \nu( F)=\int_U \nu(u,F)d\mu(u)$, $F\in \cal V$.
If $\nu$ is a $\sigma$-finite measure on $(V,\cal V)$, and $\mu$ is a $\sigma$-finite $\nu$-kernel from $(V,\cal V)$ to $(U,\cal U)$, then we use $\mu\overleftarrow{\otimes}\nu$ to denote the pushforward measure on ${\cal U}\times{\cal V}$ of $\nu\otimes \mu$ under the map $(v,u)\mapsto (u,v)$.

The killing map $\K:(f,r)\mapsto \K_r(f)$ is also measurable. If $\nu$ is a $\sigma$-finite $\mu$-kernel from $\Sigma$ to $(0,\infty)$, we use $\K_\nu(\mu)$ to denote the pushforward measure of $\mu\otimes \nu$ under  $\K$.

Let $\Sigma_t=\{f\in\Sigma:T_f>t\}$, $t\ge0$. Then $(\Sigma_t)$ is a decreasing family of subspaces of $\Sigma$ with $\Sigma_0=\Sigma$ and $\Sigma_\infty:=\cap_{t=0}^\infty \Sigma_t=C([0,\infty))$. We will be interested in the restriction of $\F^0_t$ to $\Sigma_t$, i.e., $\F_t^0\cap\Sigma_t$. Note that $\bigcup_{0\le t<\infty} \F^0_t\cap\Sigma_t$ is a $\pi$-system, which generates the $\sigma$-algebra $\F^0$, and $\Sigma=\Sigma_0\in\F_0^0\cap\Sigma_0$. This enables us to apply Dynkin's  $\pi-\lambda$ theorem. For example, if two finite measures on $(\Sigma,\F^0)$ agree on each $\F_t^0\cap\Sigma_t$, then they are equal. 

Let $\mu$ and $\nu$ be two  measures on $(\Sigma,\F^0)$, which are $\sigma$-finite on $\F_0^0$ (and so are $\sigma$-finite on each $\F_t^0$). We say that $\nu$ is locally absolutely continuous w.r.t.\ $\mu$, and write $\nu\tl \mu$, if for any $0\le t<\infty$, $\nu|_{ \F_t^0\cap\Sigma_t}$ is absolutely continuous w.r.t.\ $\mu|_{\F_t^0\cap\Sigma_t}$.
This is certainly the case if $\nu\ll\mu$, i.e., $\nu$ is (globally) absolutely continuous w.r.t.\ $\mu$. The process $M_t:=\frac{d\nu|_{\F_t^0\cap\Sigma_t}}{d\mu|_{\F_t^0\cap\Sigma_t}}$, $0\le t<\infty$, is called the local Radon-Nikodym derivative of $\nu$ w.r.t.\ $\mu$. By Dynkin's  $\pi-\lambda$ theorem, we see that $\nu$ is determined by $\mu$ and $(M_t)$. Thus, we say that $\nu$ can be obtained by locally weighting $\mu$ by $(M_t)$. 

\begin{Proposition}
Let $\mu$ be a measure on $(\Sigma,\F^0)$, which is $\sigma$-finite on $\F_0^0$.   Let $(\Upsilon,{\cal G})$ be a  measurable space. Let $\nu:\Up\times \F^0\to[0,\infty]$ be such that for every $\up\in\Up$, $\nu(\up,\cdot)$ is a finite measure on $\F^0$ that is locally absolutely continuous w.r.t.\ $\mu$. Moreover, suppose that the local Radon-Nikodym derivatives are  equal to $(M_t(\up,\cdot))$, where $M_t:(\Up,{\cal G})\times (\Sigma,\F_t)\to[0,\infty)$ is measurable for every $t\ge 0$. Then $\nu$ is a kernel from $(\Upsilon,{\cal G})$ to $(\Sigma,\F^0)$. Moreover, if $\xi$ is a $\sigma$-finite measure on $(\Upsilon,{\cal G})$ such that $\mu$-a.s., $\int_\Up M_t(\up,\cdot)d\xi(\up)<\infty$ for all $t\ge 0$, then $\xi \cdot \nu\tl\mu$, and the local Radon-Nikodym derivatives are $\int_\Up M_t(\up,\cdot)d\xi(\up)$, $0\le t<\infty$. \label{Prop-fubini}
\end{Proposition}
\begin{proof} By Dynkin's  $\pi-\lambda$ theorem, to prove that $\upsilon\mapsto \nu(\up,\cdot)$ is measurable, it suffices to show that, for any $t\in[0,\infty)$ and any $A\in\F_t^0\cap\Sigma_t$, $\up\mapsto  \nu(\up,A)$ is measurable, which easily follows from Tonelli's theorem because $\nu(\up, A)=\int_{A} M_t(\up,f)d\mu(f)$. To prove that $\xi\cdot\nu\tl\mu$ and find the local Radon-Nikodym derivatives, we apply Tonelli's theorem again and get
$$\xi\cdot\nu(A)=\int_\Up \nu(\up, A) d\xi(\up)=\int_\Up \int_{A} M_t(\up,f)d\mu(f)d\xi(\up)
= \int_{A}\int_\Up M_t(\up,f)d\xi(\up)d\mu(f)$$ for $A\in\F_t^0\cap\Sigma_t$.
Then we conclude that $\xi\cdot\nu\tl\mu$, and conclude that the local Radon-Nikodym derivatives are $\int_\Up M_t(\up,\cdot)d\xi(\up)$, $0\le t<\infty$. \end{proof}


\begin{Proposition}
 Let $\mu$ be a probability measure on $(\Sigma,\F^0)$. Let $\xi$ be a $\mu$-kernel from $(\Sigma,\F^0)$ to $(0,\infty)$ that satisfies $\EE_\mu[ |\xi|]<\infty$. Then $\K_{\xi}(\mu)\tl \mu$, and the local Radon-Nikodym derivatives are $\EE_\mu[\xi((t,\infty))| \F^0_t]$, $0\le t<\infty$.
\label{Prop-trunc}
\end{Proposition}
\begin{proof}
  Fix $t\in[0,\infty)$, and $E \in\F^0_{t}\cap\Sigma_t$. It is easy to see that $\K_r(f)\in E$  iff $f\in E$ and $r>t$,
  which implies that
  $$\K_{\xi}(\mu)(E) =\mu\otimes \xi(E\times (t,\infty))=\int_E \xi(f,(t,\infty))d\mu(f)= \int_E \EE_\mu[\xi((t,\infty))|\F^0_{t}]d\mu.$$ Then we get the conclusion. \end{proof}

\no{\bf Remark.} 
Proposition \ref{Prop-trunc} will be mainly applied to the case that $\xi=d\theta$, where $(\theta_t)$ is an $(\F_t^\mu)$-adapted right-continuous increasing process with $\theta_0=\theta_{0^-}=0$ and $\EE_\mu[\theta_\infty]<\infty$.  Applying the proposition, we find that $\K_{d\theta}(\mu)\tl \mu$, and
\BGE \mu-\mbox{a.s.},\quad \frac{d\K_{d\theta}(\mu)|_{\F^0_t\cap\Sigma_t}}{d\mu|_{\F^0_t\cap\Sigma_t}} =\EE_\mu[\theta_\infty| \F_t^\mu]-\theta_t.\label{trunc'}\EDE

Fix $\kappa>0$. Let $\PP_\kappa$ be the law of $\sqrt\kappa$ times a standard Brownian motion. This means that $\frac 1{\sqrt\kappa}$ times the coordinate process on $\Sigma$ under $\PP_\kappa$ is a standard Brownian motion. For this reason, we use $(B_t)$ to denote the above standard Brownian motion on $\Sigma$, i.e., $\frac1{\sqrt\kappa}$ times the coordinate process. 
We observe that $\PP_\kappa$ is supported by $\Sigma_\infty\cap\Sigma_\oplus$.
Let $\F^B_t$ and $\F^B$ be the $\PP_\kappa$-completion of $\F_t^0$ and $\F^0$, respectively; and
 $\EE_\kappa$ denote the expectation w.r.t.\ $\PP_\kappa$.

We now use Girsanov's theorem to derive local Radon-Nikodym derivatives. Recall that when we used Girsanov's theorem to weight a probability measure by a positive local martingale, we had to stop the process at some stopping time to get a bounded martingale. The following proposition says that we do not need to do the stopping, and the local martingale valued at different times are just the local Radon-Nikodym derivatives.

\begin{Proposition}
  Suppose that $(X_t)_{0\le t<T_0}$ satisfies $X_0=0$ and the $(\F^B_t)$-adapted SDE:
  $$ dX_t=\sqrt\kappa d B_t+ \sigma_t dt ,\quad 0\le t<T_0,$$ 
  where $T_0$ is a positive $(\F^B_t)$-stopping time and $(\sigma_t)_{0\le t<T_0}$ is a real valued $(\F^B_t)$-adapted continuous process.
  Let $\PP^{\kappa,\sigma}$ denote the law of $X$. Then $\PP^{\kappa,\sigma}\tl\PP_\kappa$. Moreover, if $M_t$, $0\le t<T_0$, is an $(\F^B_t)$-adapted continuous  local martingale that satisfies $M_0=1$ and the SDE:
  \BGE dM_t=M_t\frac{\sigma_t}{\sqrt\kappa} dB_t,\quad 0\le t<T_0.\label{dMt}\EDE
  then
  \BGE  \frac{d\PP^{\kappa,\sigma}|_{\F^0_t\cap \Sigma_t}}{d\PP_\kappa|_{\F^0_t\cap\Sigma_t}} ={\bf 1}_{T_0>t } M_t,\quad 0\le t<\infty.\label{dX/dY}\EDE \label{Prop-Girs}
\end{Proposition}
\begin{proof}
Let $M_t=\exp(\int_0^t \frac{\sigma_s}{\sqrt\kappa}dB_s-\int_0^t\frac{\sigma_s^2}{2\kappa}ds)$, $0\le t<T_0$. From It\^o's formula (cf.\ \cite{RY}), we see that $(M_t)$ satisfies $M_0=1$ and (\ref{dMt}). Thus, it suffices to prove (\ref{dX/dY}).

Fix $N\in\N$ and let $\tau_N=\inf\{0\le t<T_0: |M_t|\ge N\}$. Here we set $\inf\emptyset=T_0$. Then $\tau_N$ is a stopping time with $\tau_N\le T_0$, and $M_t$, $0\le t<\tau_N$, is uniformly bounded by $N$. If $\tau_N=T_0$, then $\PP_\kappa$-a.s.\ $\lim_{t\to T_0^-} M_t$ exists.
  Let
$$M^{\tau_N}_t:=\left\{\begin{array}{ll} M_{t\wedge \tau_N},&t\wedge\tau_N<T_0;\\  \liminf_{s\to T_0^-} M_s, & t\ge\tau_N=T_0.
\end{array}
\right.$$
Then $(M^{\tau_N}_t,0\le t<\infty)$ is a uniformly bounded ($\PP_\kappa$-a.s.) continuous local martingale, and satisfies the SDE:
\BGE dM^{\tau_N}_t=M^{\tau_N}_t{\bf 1}_{t<\tau_N}\frac{\sigma_t}{\sqrt\kappa} dB_t,\quad 0\le t<\infty.\label{dM-tau}\EDE
Let $M^{\tau_N}_\infty=\liminf_{t\to\infty} M^{\tau_N}_t$. Then for any $0\le t<\infty$, a.s.\ $\EE_\kappa[M^{\tau_N}_\infty|\F^B_t]=M^{\tau_N}_t$.
In particular, since $M^{\tau_N}_0=1$, we have $\EE_\kappa[M^{\tau_N}_\infty]=1$. Define $\PP^{\kappa,\sigma}_N$ such that $\frac{d \PP^{\kappa,\sigma}_N}{d\PP_\kappa}=M^{\tau_N}_\infty$. Then $\PP^N_\kappa$ is also a probability measure on $(\Sigma,\F^0)$. Let
$$B^N_t=B_t-\int_0^{\tau_N\wedge t} \frac{\sigma_s}{\sqrt\kappa} ds,\quad 0\le t<\infty.$$
From Girsanov's theorem (cf.\ \cite{RY}) and (\ref{dM-tau}), we know that the law of $B^N$ under $\PP^{\kappa,\sigma}_N$ is also that of a standard Brownian motion.
From
$$\sqrt\kappa B_t=\sqrt\kappa  B^N_t+\int_0^{ t}  {\sigma_s}  ds,\quad 0\le t<\tau_N.$$
we see that the law of $(\sqrt\kappa B_t,0\le t<\tau_N)$ under $\PP^{\kappa,\sigma}_N$ is the same as the law of $(X_t,0\le t<\tau_N)$ under $\PP_\kappa$.

Fix $t\in[0,\infty)$ and  $E\in\F_{t}^0\cap\Sigma_t$. Since $T_0$ is the lifetime of $X$, and $T_0=\sup_{N\in\N}\tau_N$, we have
\BGE X^{-1}(E)\subset X^{-1}(\Sigma_{t})\subset \{T_0>t\}=\bigcup_{N\in\N}\{\tau_N>t\}.\label{YE}\EDE
Since the law of $(\sqrt\kappa B_t,0\le t<\tau_N)$ under $\PP^{\kappa,\sigma}_N$ is the same as the law of $(X_t,0\le t<\tau_N)$ under $\PP_\kappa$, and $(\sqrt\kappa B_t)$ is the coordinate process, we get
$$\PP_\kappa[X^{-1}(E)\cap \{\tau_N>t\}]=\PP^{\kappa,\sigma}_N[E\cap \{\tau_N>t\}]=\EE_\kappa[{\bf 1}_{E\cap \{\tau_N>t\}}M^{\tau_N}_\infty]$$ $$=\EE_\kappa[{\bf 1}_{E\cap \{\tau_N>t\}}M^{\tau_N}_{t}]=\EE_\kappa[{\bf 1}_{E}{\bf 1}_ {\{\tau_N>t\}}M_{t}],$$
where the third equality follows from the optional stopping theorem, and the last equality holds because $M^{\tau_N}_{t}=M_{t}$ on $\{\tau_N>t\}$. This together with (\ref{YE}) implies that
$$\PP^{\kappa,\sigma}(E)=\PP_\kappa [X^{-1}(E)]=\lim_{N\to\infty}\PP_\kappa[X^{-1}(E)\cap \{\tau_N>t\}]= \EE_\kappa[{\bf 1}_{E}{\bf 1}_ {\{T_0>t\}}M_{t}].$$
So we get (\ref{dX/dY}) and finish the proof. \end{proof}

At the end of this section, we state and prove the following proposition, which extends the strong Markov property of Brownian motions.

\begin{Proposition} Let $(\theta_t)_{0\le t<\infty} $ be a right-continuous increasing $(\F^B_t)$-adapted process that satisfies $\theta_0=\theta_{0^+}=0$ and $\EE_\kappa[\theta_\infty]<\infty$. Then
  \BGE {\K_{d\theta}(\PP_\kappa)\ha\oplus \PP_\kappa}=\PP_\kappa\otimes d\theta.\label{Xoplus}\EDE \label{Prop-Pkdecmp}
Thus, $\K_{d\theta}(\PP_\kappa)\oplus \PP_\kappa\ll\PP_\kappa$, and $\theta_\infty$ is the Radon-Nikodym derivative.
\end{Proposition}

\no{\bf Remark.} If $\theta_t={\bf 1}_{\tau\le t}$, where $\tau$ is a positive finite $( \F^B_t)$-stopping time, then the proposition reduces to the  strong Markov property of $(\sqrt\kappa B_t)$, i.e., $\K_{\delta_\tau}(\PP_\kappa)\oplus \PP_\kappa=\PP_\kappa$.
\begin{proof}
  First, assume that there is $t_0\in(0,\infty)$ and $E\in\F^B_{t_0}$ such that $\theta_t(f)={\bf 1}_E(f)\cdot {\bf 1}_{[t_0,\infty)}(t)$. Fix $t_0'\in(0,t_0)$, $A\in\F^B_{t_0'}$ and $B\in\F^B$. For every $r\in[0,\infty)$, define ${\cal S}_r:\Sigma_r\to\Sigma_\oplus$ such that if $g={\cal S}_r(f)$, then $T_g=T_f-r$, and $g(t)=f(r+t)-f(r)$, $0\le t<T_g$.  Let $A\oplus_{t_0} B=\{f\in A\cap\Sigma_{t_0}: {\cal S}_{t_0}(f)\in B\}$. Since ${\cal S}_{t_0}({\cal K}_{t_0}(f)\oplus g)=g$, and ${\cal K}_{t_0}(f)\in A$ iff $f\in A$ and $r>t_0'$, we get
  $$\K_{d\theta}(\PP_\kappa)\oplus \PP_\kappa(A\oplus_{t_0} B)=\K_{d\theta}(\PP_\kappa)(A) \PP_\kappa(B)=\PP_\kappa\otimes d\theta(A\times (t_0',\infty))\PP_\kappa(B)=\PP_\kappa(A\cap E) \PP_\kappa(B).$$
   From the Markov property of $(\sqrt\kappa B_t)$, we get
   $$\int_{A\oplus_{t_0} B}\theta_\infty d\PP_\kappa=\int_{A\oplus_{t_0} B} {\bf 1}_Ed\PP_\kappa=\PP_\kappa({(A\cap E)\oplus_{t_0} B})=\PP_\kappa(A\cap E)\cdot \PP_\kappa(B).$$
   Define $\PP_\kappa^\theta$ such that $d\PP_\kappa^{\theta}/d\PP_\kappa=\theta_\infty$. From the above two displayed formulas, we see that $\K_{d\theta}(\PP_\kappa)\oplus \PP_\kappa$ agrees with $\PP_\kappa^{\theta}$ on the sets $A\oplus_{t_0} B$, where $A\in\F^B_{t_0'}$, $t_0'\in(0,t_0)$, and $B\in\F^B$. Since these sets form a $\pi$-system, Dynkin's  $\pi-\lambda$ theorem implies that the two measures agree on the $\sigma$-algebra generated by these sets, which agrees with  $\F^B$ restricted to $\Sigma_{t_0}$. Since both measures are supported by $\Sigma_\infty\subset\Sigma_{t_0}$,  we get $\K_{d\theta}(\PP_\kappa)\oplus \PP_\kappa=\PP_\kappa^{\theta}$. Since these two measures are the projections of $\K_{d\theta}(\PP_\kappa)\ha\oplus \PP_\kappa$ and $\PP_\kappa\otimes d\theta$, respectively, to $\Sigma$, and the projections of them to $(0,\infty)$ are both concentrated at $t_0$, we get (\ref{Xoplus}) in this special case.

   Second, assume that $(\theta_t)$ has the form of $\sum_{n=1}^\infty C_n\theta^{(n)}_t$, where each $C_n$ is a nonnegative real number and each $\theta^{(n)}_t$ satisfies the condition in the previous paragraph. In this case, we get (\ref{Xoplus}) using the result in the above paragraph and the fact that both sides of (\ref{Xoplus}) satisfy the countable linearity in $(\theta_t)$.

   Finally, we consider the general case. Since $\EE_\kappa[\theta_\infty]<\infty$, from the linearity of both sides of (\ref{Xoplus}) in $d\theta$, we may assume that  $\EE_\kappa[\theta_\infty]=1$. In this case both sides of (\ref{Xoplus}) are probability measures.
   For $n\in\N$, define $U^{(n)}$ and $L^{(n)}$ from $[0,\infty)$ to $[0,\infty)$ such that
   $$U^{(n)}(t)=\frac{\lfloor 2^n\cdot t\rfloor+1}{2^n},\quad L^{(n)}(t)=0\vee \frac{\lceil 2^n\cdot t\rceil -1}{2^n}.$$ Then $U^{(n)}(t)\downarrow t$ and $L^{(n)}(t)\uparrow t$ for any $t\in[0,\infty)$. Moreover, we have
   \BGE  L^{(n)}(t)\le s\quad \mbox{if and only if} \quad t\le U^{(n)}(s),\quad \forall t,s\in[0,\infty).\label{equiv}\EDE
   This equivalence holds because for any $t,s\ge 0$, both sides are equivalent to that there is an integer $n$ such that $2^nt\le n\le 2^ns+1$. Define $(\theta^{(n)}_t)$ such that $\theta^{(n)}_t=\theta_{U^{(n)}(t)}$. Then $(\theta^{(n)}_t)$ has the form of that in the above paragraph since it takes values only at $k/2^n$, $k\in\Z$. Thus, $\K_{d\theta^{(n)}}(\PP_\kappa)\ha\oplus \PP_\kappa=\PP_\kappa\otimes d\theta^{(n)}$ for each $n\in\N$. From (\ref{equiv}), we get
   \BGE L^{(n)}_*(d\theta(f,\cdot))=d\theta^{(n)}(f,\cdot),\quad \forall f\in\Sigma, n\in\N.\label{push}\EDE

   We assign $\Sigma$ the topology of locally uniform convergence. It suffices to show that $\K_{d\theta^{(n)}}(\PP_\kappa)\ha\oplus \PP_\kappa$ and $\PP_\kappa\otimes d\theta^{(n)}$ converge weakly to $\K_{d\theta}(\PP_\kappa)\ha\oplus \PP_\kappa$ and $\PP_\kappa\otimes d\theta$, respectively. To prove that $\K_{d\theta^{(n)}}(\PP_\kappa)\ha\oplus \PP_\kappa\to \K_{d\theta}(\PP_\kappa)\ha\oplus \PP_\kappa$, we define $\Sigma\times(0,\infty)$-valued random variables $h^{(n)}$ and $h$ on a probability space $(\Omega,\PP)$ such that their  laws are the above measures, and a.s.\ $h^{(n)}\to h$. For this purpose, we choose $\Omega=(\Sigma_\infty\times (0,\infty))\times \Sigma_\oplus$, $\PP=(\PP_\kappa\otimes d\theta)\times \PP_\kappa$, $h((f,t),g)=(\K_t(f)\oplus g,t)$, and $h^{(n)}(f,t,g)=h((f,L^{(n)}(t)),g)$. Using (\ref{push}) and the convergence $L^{(n)}(t)\uparrow t$ it is easy to check that $h^{(n)}$ and $h$ satisfy the desired properties. To prove that $\PP_\kappa\otimes d\theta^{(n)}\to\PP_\kappa\otimes d\theta$, we choose $\Omega=\Sigma\times(0,\infty)$,  $\PP=\PP_\kappa\otimes d\theta$, $h=\id_\Omega$, and $h^{(n)}(f,t)=(f,L^{(n)}(t))$. Then a.s.\ $h^{(n)}\to h$, and from (\ref{push}) we see that the laws of $h^{(n)}$ and $h$ are $\PP_\kappa\otimes d\theta^{(n)}$ and $\PP_\kappa\otimes d\theta$, respectively. So we get (\ref{Xoplus}) in the general case.

   The statement after (\ref{Xoplus}) follows from projecting both sides of (\ref{Xoplus}) to $\Sigma$.
\end{proof}

\section{Schramm-Loewner Evolution} \label{section-SLE}
In this section, we review the  Loewner equations and the Schramm-Loewner Evolution (SLE). See \cite{Law1,RS} for more details. We focus on chordal SLE, and will often omit the word ``chordal'' before ``Loewner equation" or ``SLE'' when there is no confusion.

The definition of SLE uses the Loewner equations. Let's first review the (chordal) Loewner equation.
Let $\lambda\in C([0,T))$, where $T\in(0,\infty]$. The Loewner equation driven by $\lambda$ is the following differential equation in the complex plane:
$$\pa g_t(z)=\frac{2}{g_t(z)-\lambda(t)},\quad 0\le t<T;\quad g_0(z)=z.$$
Let $\HH=\{z\in\C:\Imm z>0\}$. For $0\le t<T$, let $K_t$ denote the set of $z\in\HH$ such that the solution $s\mapsto g_s(z)$ blows up before or at $t$. It turns out that $g_t$ maps $\HH\sem K_t$ conformally onto $\HH$, and satisfies $g_t(z)=z+\frac{2t}z+O(|z|^{-2})$ as $z\to\infty$. We call $g_t$ and $K_t$ the Loewner maps and hulls, respectively, driven by $\lambda$.

Suppose for every $t\in[0,T)$, $g_t^{-1}$ extends continuously from $\HH$ to $\lin\HH$. Throughout, we use $f_t$ to denote the continuation of $g_t^{-1}$ from $\lin\HH$ into $\lin\HH$. Also suppose that $\gamma(t):=f_t(\lambda(t))$, $0\le t<T$, is a continuous curve in $\lin\HH$. Then we say that $\gamma$ is the Loewner curve driven by $\lambda$. In this case, for $0\le t<T$, $\HH\sem K_t$ is the unbounded connected component of $\HH\sem\gamma([0,t])$. 
The Loewner curve driven by $\lambda$ may not exist in general.

The Loewner equations satisfy the following scaling and translation properties. Suppose $\lambda(t)$, $0\le t<T$, generates Loewner maps $g_t$ and hulls $K_t$, $0\le t<T$. Let $a>0$ and $b\in\R$, and $\lambda^{a,b}(t)=b+a\cdot\lambda(t/a^2)$, $0\le t<a^2T$. Then $\lambda^{a,b}$ generates the Loewner maps $z\mapsto b+a\cdot g_{t/a^2}(({z-b})/a)$ and hulls $b+a\cdot K_{t/a^2}$, $0\le t<a^2T$. If $\lambda$ generates a Loewner curve $\gamma$, then $\lambda^{a,b}$ also generates a Loewner curve, which is $b+a\cdot \gamma({\cdot/a^2})$.

Another simple and useful property of the Loewner equations is the renewal property. Suppose $\lambda(t)$, $0\le t<T$, generates Loewner maps $g_t$ and hulls $K_t$, $0\le t<T$. Let $\tau\in[0,T)$. Then $\lambda(\tau+t)$, $0\le t<T-\tau$, generates Loewner maps $g_{\tau+t}\circ g_\tau^{-1}$ and hulls $g_\tau(K_{\tau+t}\sem K_\tau)$, $0\le t<T-\tau$.  If $\lambda$ and $\lambda(\tau+\cdot)$ generate Loewner curves $\gamma$ and $\gamma_\tau$, respectively, then $\gamma(\tau+t)=f_\tau(\gamma_\tau(t))$, $0\le t<T-\tau$.

Let $\Sigma^\C$ denote the counterpart of $\Sigma$ with real valued continuous functions replaced by complex valued continuous functions. Let $\Sigma^\Lo$ denote the set of driving functions $\lambda\in\Sigma$ that generate a Loewner curve $\gamma$. Then $\Sigma^\Lo\in\F^0$, and the Loewner map $\Lo:\lambda\mapsto \gamma$ from $\Sigma^\Lo$ to $\Sigma^\C$ is measurable. We also define here the extended Loewner map $\ha\Lo$ from $\{(\lambda,t):\lambda\in\Sigma^\Lo,0\le t<T_\lambda\}$ to $\Sigma^\C\times\lin\HH$ such that $\ha\Lo(\lambda,t)=(\Lo(\lambda),\Lo(\lambda)(t))$.

Fix $\kappa>0$. Let $B(t)$, $0\le t<\infty$, be a standard Brownian motion. The SLE$_\kappa$ process is defined by taking $\lambda(t)=\sqrt\kappa B(t)$, $0\le t<\infty$, in the Loewner equation. In this case, the Loewner curve $\gamma$ driven by $\lambda$ a.s.\ exists, and satisfies $\gamma(0)=0$ and $\lim_{t\to\infty}\gamma(t)=\infty$. Such $\gamma$ is called a standard SLE$_\kappa$ curve (in $\HH$ from $0$ to $\infty$). In terms of measures, this means that $\PP_\kappa$ (the law of $(\sqrt\kappa B_t)$) is supported by $\Sigma^\Lo$. The pushforward measure  $\Lo_*(\PP_\kappa)$ is then the law of a standard SLE$_\kappa$ curve.

The scaling property of Loewner equations and the scaling property of Brownian motions together imply the scaling property of the SLE curve: if $\gamma$ is a standard SLE$_\kappa$ curve, then $t\mapsto a\gamma(t/a^2)$ is also a standard SLE$_\kappa$ curve. The renewal and translation properties of Loewner equations and the strong Markov property of Brownian motions together imply the domain Markov property of SLE: if $\gamma$ is a standard SLE$_\kappa$ curve, and $\tau$ is a finite stopping time, then conditioned on $\gamma(t)$, $t\le \tau$, there is a standard SLE$_\kappa$ curve $\ha\gamma$ such that $\gamma(\tau+t)= f_\tau(\ha\gamma(t)+\lambda_\tau)$, $t\ge 0$.

The definition of SLE$_\kappa$ extends to other simply connected domains by conformal maps. Let $D$ be a simply connected domain with locally connected boundary. Let $a$ and $b$ be two distinct prime ends (\cite{Ahl}) of $D$. Let $f$ be a conformal map from $\HH$ onto $D$ such that $f(0)=a$ and $f(\infty)=b$. Let $\gamma$ be a standard SLE$_\kappa$ curve. Then $f\circ \gamma$ is called an SLE$_\kappa$ curve in $D$ from $a$ to $b$. The local connectedness of $\pa D$ is used to guarantee that $f$ extends continuously to $\lin\HH$ so that $f\circ \gamma$ is a continuous curve in $\lin D$. This condition may be weakened in some cases.
Although the $f$ is not unique, the law of $f\circ \gamma$ is unique up to a linear time-change, thanks to the scaling property of a standard SLE$_\kappa$ curve. From now on, an SLE$_\kappa$ curve without the domain and two prime ends specified is always a standard SLE$_\kappa$ curve, and the word ``standard'' will often be omitted

The behavior of an SLE$_\kappa$ curve depends on the value of $\kappa$. Let $\gamma$ be an SLE$_\kappa$ curve. If $\kappa\in(0,4]$, then $\gamma$ is a simple curve, and does not intersect $\R$ after the time $0$; if $\kappa>4$, then $\gamma$ will intersect itself and $\R$ after the time $0$. If $\kappa\ge 8$, $\gamma$ is space-filling: it visits every point in $\lin\HH$; if $\kappa<8$, then for every $z_0\in\lin\HH\sem\{0\}$, the probability that $\gamma$ visits $z_0$ is $0$. Moreover, the Hausdorff dimension of $\gamma$  is $\min\{1+\frac\kappa 8,2\}$ (cf.\ \cite{dim-SLE}).



SLE$_\kappa(\rho)$ is a variant of SLE$_\kappa$. Its definition involves one or more force points, which lie on the boundary or in the interior of the domain. For the purpose here, we consider the case that there is only one force point, which is an interior point. Let $\rho\in\R$, $a_0\in\R$ and $z_0\in\HH$. An SLE$_\kappa(\rho)$ process started from $a_0\in\R$ with force point at $z_0\in\HH$ is the solution of the Loewner equation driven by $\lambda_t$, $0\le t<T_{z_0}$, which is the solution of the SDE:
$$ d\lambda_t=\sqrt\kappa dB_t+\Ree \frac{\rho}{\lambda_t-g^{\lambda }_t(z_0)}\,dt,\quad \lambda_0=a_0. $$
Here $g^{\lambda }_t$ are the Loewner maps driven by $\lambda $, and $[0,T_{z_0})$ is the maximal solution interval.
Let $\PP^{\kappa,\rho}_{z_0}$ denote the law of $\lambda_t$, $0\le t<T_{z_0}$. From Proposition \ref{Prop-Girs}, we know that $\PP^{\kappa,\rho}_{z_0}\tl\PP_\kappa$. Thus, $\PP^{\kappa,\rho}_{z_0}$ is supported by $\Sigma^\Lo$, i.e., the Loewner curve $\gamma$ driven by $\lambda$ a.s.\ exists, which is called a (standard) SLE$_\kappa(\rho)$ curve started from $a_0$ with force point at $z_0$. The following proposition is a result of \cite{MS4}.


\begin{Proposition}
 For any $\kappa>0$ and $\rho\le \frac\kappa 2-4$,
 \begin{enumerate}
   \item[(i)] a.s.\ $T_{z_0}<\infty$ and $\lim_{t\to T_{z_0}^-} \lambda_t\in\R$;
   \item[(ii)] a.s.\ $\lim_{t\to T_{z_0}^-} \gamma(t)=z_0$.
 \end{enumerate}
 \label{transience-0}
\end{Proposition}

From the above proposition we see that, if $\rho\le \frac\kappa 2-4$, then $\PP^{\kappa,\rho}_{z_0}$ is supported by $\Sigma^\oplus$, and we may define $\PP^{\kappa,\rho}_{z_0}\oplus \PP_\kappa$, which is supported by $\Sigma_\Lo$. The pushforward measure $\Lo_*(\PP^{\kappa,\rho}_{z_0}\oplus \PP_\kappa)$ is called the law of a standard extended SLE$_\kappa(\rho)$ curve through $z_0$, which is supported by the continuous curves in $\lin\HH$ from $0$ to $\infty$ that pass through $z_0$. In other words, a standard extended SLE$_\kappa(\rho)$ curve through $z_0$ is defined by continuing  an SLE$_\kappa(\rho)$ curve started from $0$ with force point at $z_0$ by an SLE$_\kappa$ curve in the remaining domain from $z_0$ to $\infty$.


Using conformal maps, we may define an extended SLE$_\kappa(\rho)$ curve (for $\rho\le \frac\kappa 2-4$) in a simply connected domain from one prime end to another prime end through an interior point. We will mainly work on extended SLE$_\kappa(\rho)$ curves in $\HH$ from $0$ to $\infty$ through some $z_0\in\HH$, and so will omit the word ``standard''. 


We now derive the local Radon-Nikodym derivative of $\PP^{\kappa,\rho}_{z_0}$ w.r.t.\ $\PP_\kappa$. Let $\lambda_t=\sqrt\kappa B_t$, and $g_t$ be the Loewner maps driven by $\lambda$. Fix $z_0=x_0+iy_0\in\HH$. For $0\le t< T_{z_0}$, let
\BGE Z_t=g_t(z_0)-\lambda_t,\quad X_t=\Ree Z_t,\quad Y_t=\Imm Z_t,\quad D_t=|g_t'(z_0)|.\label{ZXYD}\EDE
From Loewner's equation, we see that   $Y_t$  and $D_t$ satisfy the ODEs:
\BGE \frac{dY_t}{Y_t}=\frac{-2}{X_t^2+Y_t^2}dt,\quad \frac{dD_t}{D_t}=\frac{-2(X_t^2-Y_t^2)}{(X_t^2+Y_t^2)^2}dt;\label{XYD}\EDE
and $X_t$ satisfies the SDE
\BGE dX_t=-\sqrt\kappa dB_t+\frac{2X_t}{X_t^2+Y_t^2}dt.\label{X}\EDE
Define
\BGE M^{\kappa,\rho}_t(z_0)=|Z_t|^{\frac\rho\kappa}\cdot Y_t^{\frac{\rho^2}{8\kappa}} \cdot D_t^{\frac\rho\kappa(1-\frac \kappa 4+\frac\rho 8)},\quad 0\le t<T_{z_0},\label{M-kappa}\EDE
and \BGE G^{\kappa,\rho}(z_0)=M^{\kappa,\rho}_0(z_0)=|z_0|^{\frac\rho\kappa}\cdot \Imm(z_0)^{\frac{\rho^2}{8\kappa}}.\label{G-kappa-rho} \EDE
Using It\^o's formula and (\ref{XYD})-(\ref{X}), it is straightforward to check that $(M^{\kappa,\rho}_t(z_0))$ is an $(\F^0_t)$-adapted continuous local martingale, and satisfies the SDE:
$$ dM^{\kappa,\rho}_t(z_0)=M^{\kappa,\rho}_t(z_0)\cdot \Ree \frac{\rho/\sqrt\kappa}{\lambda_t-g_t(z_0)}\,dB_t,\quad 0\le t<T_{z_0}.$$ 
We further define
\BGE M^{\kappa,\rho}_t(z_0)=0,\quad T_{z_0}\le t<\infty.\label{M-kappa2}\EDE
For any measurable subset $U$ of $\lin\HH$, define
\BGE \Psi^{\kappa,\rho}_t(U)=\int_{U\cap\HH}  M^{\kappa,\rho}_t(z)dA(z).\label{Psi-U}\EDE
Throughout this paper, we use $dA$ to denote the Lebesgue measure on $\C$. Since $(M^{\kappa,\rho}_t(z))$ is a positive local martingale, it is also a supermartingale. Thus, if $\int_U G^{\kappa,\rho}(z)dA(z)<\infty$, then $(\Psi^{\kappa,\rho}_t(U))$ is also a supermartingale, which has to be $\PP_\kappa$-a.s.\ finite.

From Proposition \ref{Prop-Girs} we know that \BGE \frac{d\PP^{\kappa,\rho}_{z_0}|_{\F^0_t\cap\Sigma_t}}{d\PP_\kappa|_{\F^0_t\cap\Sigma_t}}= \frac{M^{\kappa,\rho}_t(z_0)}{G^{\kappa,\rho} (z_0)},\quad 0\le t<\infty.\label{RN-rho}\EDE
From Proposition \ref{Prop-fubini}, $(z, E)\mapsto \PP^{\kappa,\rho}_z(E)$ is a probability kernel from $\HH$ to $(\Sigma,\F^0)$. Thus, for any measurable subset $U$ of $\lin\HH$, we may define the measure
\BGE \PP^{\kappa,\rho}_U=\int_{U\cap\HH} \PP^{\kappa,\rho}_z G^{\kappa,\rho}(z)dA(z).\label{P-kappa-rho-U}\EDE
Moreover, if $\int_U G^{\kappa,\rho}(z)dA(z)<\infty$, then $\PP^{\kappa,\rho}_U\tl \PP_\kappa$, and
\BGE \frac{d\PP^{\kappa,\rho}_U|_{\F^0_t\cap\Sigma_t}}{d\PP_\kappa|_{\F^0_t\cap\Sigma_t}}=\Psi^{\kappa,\rho}_t(U),\quad 0\le t<\infty.\label{dP-kappa-rho-U}\EDE

\section{Natural Parametrization} \label{section-NP}
Fix $\kappa\in(0,8)$. Let $d=1+\frac\kappa 8$ be the Hausdorff dimension of SLE$_\kappa$ curves. We will review the natural parametrization of SLE$_\kappa$ in this section.

First, we review the definition of the SLE Green's function. The Green's function for an SLE$_\kappa$ curve $\gamma$ in a simply connected domain $D$ from one prime end $a$ to another prime end $b$ is
$$G_{(D;a,b)}(z):= \lim_{r\to 0^+} r^{d-2} \PP[\dist(z,\gamma)<r],\quad z\in D,$$
provided that the limit exists. It is clear that, if the SLE$_\kappa$ Green's function exists for one configuration $(D;a,b)$, then it exists for all configurations, and it satisfies the conformal covariance, i.e., if $f$ maps $(D;a,b)$ conformally onto $(D';a',b')$, then
$$G_{(D;a,b)}(z)=|f'(z)|^{2-d}G_{(D';a',b')}(f(z)).$$
It is proved in  \cite{LR} that there is an unknown positive constant $C_\kappa>0$ depending only on $\kappa$ such that
\BGE G_{(\HH;0,\infty)}(z)=C_\kappa|z|^{d-2}\sin^{\frac\kappa 8+\frac8\kappa-2}(\arg z),\quad z\in\HH,\label{Green-lim}\EDE
We will write $G(z)$ for $G_{(\HH;0,\infty)}(z)$. Define
$$M_t(z)=|g_t'(z)|^{2-d}G(g_t(z)-\lambda(t)),\quad 0\le t<T_z.$$ and $M_t(z)=0$ for $t\ge T_z$. For any measurable set $U\subset\lin\HH$, define $\Psi_t(U)=\int_{U\cap \HH} M_t(z)dA(z)$. It is easy to check that $G(z)$, $M_t(z)$, $\Psi_t(U)$ agree with $C_\kappa$ times $G^{\kappa,\kappa-8}(z)$, $M^{\kappa,\kappa-8}_t(z)$, $\Psi^{\kappa,\kappa-8}_t(U)$, respectively, defined by (\ref{M-kappa})-(\ref{Psi-U}). From (\ref{RN-rho}) we see that locally weighting $\sqrt\kappa B_t$ using $(M_t(z_0)/G(z_0))$ generates a driving  SLE$_\kappa(\kappa-8)$ process with force point at $z_0$.

A two-sided radial SLE$_\kappa$ curve is just an extended SLE$_\kappa(\kappa-8)$ curve. This means that the law of  a (standard) two-sided radial SLE$_\kappa$ curve through $z_0\in\HH$ can be expressed by $\Lo_*(\PP^{\kappa,\kappa-8}_{z_0}\oplus \PP_\kappa)$. Two-sided radial SLE is important because it can be understood as an SLE$_\kappa$ curve conditioned to pass though an interior point. To make this rigorous, one may condition an SLE$_\kappa$ curve $\gamma$ on the event that $\dist(z_0,\gamma)<r$ for a fixed interior point $z_0$, and then pass to the limit $r\to 0$.

From the reversibility of chordal SLE$_\kappa$ (cf.\ \cite{reversibility,MS3}) it is easy to see that  two-sided radial SLE$_\kappa$  also  satisfies  reversibility, i.e., the time-reversal of a two-sided radial SLE$_\kappa$ curve in a simply connected domain $D$ from   $a$ to   $b$ through  $z_0$ agrees with a two-sided radial SLE$_\kappa$ curve in  $D$ from $b$ to $a$ through $z_0$, up to a reparametrization. This property is not satisfied by extended SLE$_\kappa(\rho)$ processes for $\rho\ne \kappa-8$.

Recall that $\PP^{\kappa,\kappa-8}_{\cdot}$ is a kernel from $\HH$ to $(\Sigma,\F^0)$. So we can defined the measure $\PP_U=\int_{U\cap\HH} \PP^{\kappa,\kappa-8}_t(z)G(z)dA(z)$ for any measurable set $U\subset \lin\HH$. Then $\PP_U$ equals $C_\kappa$ times the  $\PP^{\kappa,\kappa-8}_U$ defined by (\ref{P-kappa-rho-U}).  From (\ref{dP-kappa-rho-U}), if $\int_U G(z)dA(z)<\infty$, then $\PP_U\tl \PP_\kappa$ and
\BGE \frac{d \PP_U|_{\F_t^0\cap\Sigma_t}}{d\PP_\kappa|_{\F_t^0\cap\Sigma_t }}=\Psi_t(U),\quad 0\le t<\infty.\label{Psi-U-RN}\EDE

It is proved by Lawler and Zhou (\cite{LZ}) that, if $U$ is a pre-compact measurable subset of $\HH$, i.e., $\lin U$ is a compact subset of $\HH$, then $(\Psi_t(U))$  is of class $\cal D$, i.e.,
$\{\Psi_T(U):T\mbox{ is a finite stopping time}\}$
is uniformly integrable. In this case they can apply the Doob-Meyer decomposition theorem to get a unique continuous increasing $(\F^B_t)$-adapted process $\Theta_t(U)$ such that $\Theta_0(U)=0$ and $M_t(U):=\Psi_t(U)+\Theta_t(U)$
is a uniformly integrable $(\F^B_t)$-martingale. This means that $M_\infty(U):=\lim_{t\to\infty}M_t(U)$ a.s.\ exists, and $\EE_\kappa[M_\infty(U)|\F^B_t]=M_t(U)$, $0\le t<\infty$. From Lemma \ref{Psi-infty}, we know that $\Psi_\infty(U)=0$. Thus,
\BGE \EE[\Theta_\infty(U)|\F^B_t]-\Theta_t(U)=\Psi_t(U),\quad 0\le t<\infty.\label{Theta-Theta}\EDE

The process $(\Theta_t(U))$ determines a $\PP_\kappa$-kernel $d\Theta (U)$. Then we may define the killing   $\K_{d\Theta(U)}(\PP_\kappa)$. From (\ref{trunc'}), we have
\BGE \PP_\kappa-a.s.,\quad \frac{d \K_{d\Theta (U)}(\PP_\kappa)|_{\F_t^0\cap\Sigma_t}}{d\PP_\kappa|_{\F_t^0\cap\Sigma_t}}=\EE[\Theta_\infty(U)|\F^B_t]-\Theta_t(U),\quad 0\le t<\infty.\label{T-Theta}\EDE
Combining (\ref{Psi-U-RN}), (\ref{Theta-Theta}) and (\ref{T-Theta}), we see that,
\BGE \PP_U= \K_{d\Theta(U)}(\PP_\kappa).\label{T=P}\EDE

It is easy to check that, if $U_1$ and $U_2$ are disjoint pre-compact measurable subsets of $\HH$, then a.s.\ $\Theta_t(U_1\cup U_2)=\Theta_t(U_1)+\Theta_t(U_2)$ for $0\le t<\infty$. Thus, the kernel $d\Theta_t(U)$ (and also $\Theta_t(U)$) is increasing in $U$. We may extend the definition of $\Theta_t(U)$ to any measurable subset $U$ of $\lin\HH$ as follows. Let $(U_n)$ be a sequence of compact subsets of $\HH$ such that $U_n$ is contained in the interior of $U_{n+1}$ for every $n\in\N$, and $\HH=\bigcup U_n$. For any measurable subset $U$ of $\lin\HH$, we first define the kernel $\mu_U=\lim_{n\to\infty} d\Theta (U\cap U_n)$, then let $\Theta_t(U)=\mu_U([0,t])$, $t\ge 0$ (so $d\Theta(U)=\mu_U$). The definition does not depend on the choice of $(U_n)$. 
In particular, $\Theta_t:=\Theta_t(\lin\HH)$ is called the natural parametrization of $\gamma$.

The $\Theta_t(U)$ is understood as the total time that the chordal SLE$_\kappa$ curve $\gamma$ spends in $U$ in the natural parametrization before  time $t$. Define the kernel ${\cal M}_U$ from $\Lo(\Sigma^\Lo)$ to $\lin\HH$   by \BGE {\cal M}_U(\gamma,\cdot):=\gamma_*(d\Theta(U))\label{MU}\EDE This is the $d$-dimensional Minkowski content measure  of $\gamma $ in $U$.

Here the Minkowski content measure of a set $S$ in a domain $U$ is a measure, say ${\cal M}$ supported by $S\cap U$, which satisfies the property that for any compact set $K\subset U$, $K\cap S$ has Minkowski content, which equals to ${\cal M}(K)$ and is finite. It satisfies the property of conformal covariance. From the work of \cite{LR}, we know that an SLE$_\kappa$ curve  for $\kappa\in(0,8)$ in $\HH$ from $0$ to $\infty$ possesses $(1+\frac{\kappa}{8})$-dimensional Minkowski content measure, which equals to the pushforward measure of the curve function of the natural parametrization (as a measure on the time interval). Using conformal covariance, we then know that any chordal SLE$_\kappa$ curve possesses Minkowski content measure in the domain that it is defined. The reader is referred to \cite[Section 2.3]{loop} for details of the notation. 

\begin{Theorem} Let $\kappa\in(0,8)$. Let $U$ be any measurable subset of $\lin\HH$. Then we have
\BGE {\PP_U\ha\oplus\PP_\kappa}=\PP_\kappa\otimes d\Theta(U),\label{pre-main-thm}\EDE
  \BGE \Lo_*(\PP^{\kappa,\kappa-8}_z \oplus\PP_\kappa)\overleftarrow\otimes{\bf 1}_U G(z)dA(z)=\Lo_*(\PP_\kappa)\otimes {\cal M}_U,\label{int-U}\EDE
where ${\cal M}_U$ is the $d$-dimensional Minkowski content measure of $\gamma\cap U$.
\label{Thm1}\end{Theorem}
\begin{proof}
If $U$ is pre-compact in $\HH$, then (\ref{pre-main-thm}) follows from (\ref{T=P}) and Proposition \ref{Prop-Pkdecmp}. For general measurable $U\subset\HH$, (\ref{pre-main-thm}) follows from the above special case and a limiting procedure.

We now apply the extended Loewner map $\ha\Lo(\lambda,t)=(\Lo(\lambda),\Lo(\lambda)(t))$ to (\ref{pre-main-thm}).
We observe that $\ha\Lo_*(\PP^{\kappa,\kappa-8}_z\ha\oplus\PP_\kappa)=\Lo_*(\PP^{\kappa,\kappa-8}_z \oplus\PP_\kappa)\otimes\delta_z$, where $\delta_z$ is the Dirac measure at $z$. Integrating the equality over $z$ against the measure ${\bf 1}_U G(z)dA(z)={\bf 1}_{U\cap \HH} G(z)dA(z)$, we get
\BGE \ha\Lo_*(\PP^{\kappa,\kappa-8}_U\ha\oplus\PP_\kappa)= \Lo_*(\PP^{\kappa,\kappa-8}_z \oplus\PP_\kappa)\overleftarrow\otimes{\bf 1}_U G(z)dA(z).\label{Lo1}\EDE

Using (\ref{MU}) and the definition of $\ha\Lo$, we find that \BGE \ha\Lo_*(\PP_\kappa\otimes d\Theta(U))=\Lo_*(\PP_\kappa)\otimes {\cal M}_U.\label{Lo2}\EDE
Then (\ref{pre-main-thm}), (\ref{Lo1}) and (\ref{Lo2}) together imply (\ref{int-U}).
\end{proof}

\begin{Corollary}
  Let $\kappa\in(0,8)$. Let $U$ be a measurable subset of $\HH$ with $\int_U G(z)dA(z)<\infty$. If we integrate the laws of two-sided radial SLE$_\kappa$ curves through $z$ against the measure ${\bf 1}_U G(z)dA(z)$, then we get a finite measure on curves, which is absolutely continuous w.r.t.\ the law of SLE$_\kappa$ curve, and the Radon-Nikodym derivative is the $d$-dimensional Minkowski content of $\gamma\cap U$. \label{Cor1}
\end{Corollary}
\begin{proof}
  This follows from projecting both sides of  (\ref{int-U}) to $\Sigma^\C$ and that $\EE_\kappa[{\cal M}_U(\gamma,\cdot)]= \EE_\kappa[\Theta_\infty(U)]=\int_U G(z)dA(z)<\infty$.
\end{proof}

\begin{Corollary}
  Let $\kappa\in(0,8)$. Suppose $( \gamma,  z)$ is a $C([0,\infty),\C)\times \lin\HH$-valued random variable with the properties that $ \gamma$ has the law of an  SLE$_\kappa$ curve, and given $ \gamma$, the law of  $  z$ is absolutely continuous w.r.t.\  the $d$-dimensional Minkowski content measure of $\gamma$. Then the law of $z$ is absolutely continuous w.r.t.\ ${\bf 1}_\HH dA(z)$, and the law of $ \gamma$ given $z$ is absolutely continuous w.r.t.\ the law of a two-sided radial SLE$_\kappa$ curve through $  z$. \label{Cor2}
\end{Corollary}
\begin{proof}
  From the assumption, we see that the law of $( \gamma,  z)$ is absolutely continuous w.r.t.\ the measure in either side of (\ref{int-U}) with $U=\HH$.
\end{proof}


\no{\bf Remarks.}
\begin{enumerate}
  \item Roughly speaking, the meaning of (\ref{int-U}) is that the following two methods generate the same measure on the space of curve-point pairs:
      \begin{enumerate}
        \item [(i)] first sample a point according to the measure ${\bf 1}_U G(z)dA(z)$, and then sample a two-sided radial SLE$_\kappa$ curve $\gamma$ through $z$;
        \item [(ii)] first sample an SLE$_\kappa$ curve $\gamma$, and then sample a point $z$ on $\gamma$ according to the $d$-dimensional Minkowski content measure of $\gamma\cap U$. Here the measure of $\gamma$ is changed after sampling $z$ because ${\cal M}_U$ is not a probability kernel.
            \end{enumerate}
  \item Corollary \ref{Cor2} means that, if we sample a random point $z$ on a chordal SLE$_\kappa$ curve $\gamma$ according to a law that is absolutely continuous w.r.t.\ the $d$-dimensional Minkowski content measure, and then condition on $z$, then the curve $\gamma$ looks ``similar'' to a two-sided radial SLE$_\kappa$ curve through $z$.
\end{enumerate}

\begin{Corollary}
Let $D$ be a simply connected domain with two distinct prime ends $a$ and $b$. Let $\PP^\kappa_{D;a,b}$ denote the law of an SLE$_\kappa$ curve in $D$ from $a$ to $b$. For any $z\in D$, let $\PP^\kappa_{D;a,b;z}$ denote the law of a two-sided radial SLE$_\kappa$ curve in $D$ from $a$ to $b$ through $z$.  For any measurable set $U\subset D$, let ${\cal M}_U (\gamma,\cdot)$ be the $d$-dimensional Minkowski content measure on $\gamma\cap U$. Then
\BGE \PP^\kappa_{D;a,b;z}\overleftarrow \otimes {\bf 1}_U G_{D;a,b}(z)dA(z)= \PP^\kappa_{D;a,b}\otimes {\cal M}_U. \label{int-U''}\EDE
Moreover, if $\int_U G_{D;a,b}(z)dA(z)<\infty$, then $\int_U \PP^\kappa_{D;a,b;z} dA(z)\ll\PP^\kappa_{D;a,b}$, and the Radon-Nikodym derivative is the $d$-dimensional Minkowski content of $\gamma\cap U$.
\label{Cor-general}
\end{Corollary}
\begin{proof}
  Formula (\ref{int-U''}) follows immediately from (\ref{int-U}), the conformal covariance of SLE$_\kappa$ Green's function,  and the transformation rule of $d$-dimensional Minkowski content measure under confomal maps. The rest follows from projecting the measures in (\ref{int-U''})  to $\Sigma^\C$.
\end{proof}

\no{\bf Remark.} The above corollary in the case $\kappa\le 4$ was proved earlier by Laurie Field  in \cite{Fie}, where he assumed that the domain $D$ is bounded and has analytic boundary.

\vskip 4mm
Inspired by Theorem \ref{Thm1} and its corollaries, we make the following definition.

\begin{Definition}
  Let $\kappa>0$ and $\rho\le \frac\kappa 2-4$. We say that SLE$_\kappa$ admits an SLE$_\kappa(\rho)$ decomposition if (\ref{pre-main-thm}) holds with $G(z)$ replaced by $G^{\kappa,\rho}(z)$ and $\PP_U$ replaced by $\PP^{\kappa,\rho}_U$, for some increasing adapted process $\Theta_t(U)$, which then implies (\ref{int-U}) with $\PP^{\kappa,\kappa-8}_z$ replaced by $\PP^{\kappa,\rho}_z$ and ${\cal M}_U$ defined by (\ref{MU}) for the new $\Theta_t(U)$.
\end{Definition}

We have shown that SLE$_\kappa$ admits an SLE$_\kappa(\kappa-8)$ decomposition for $\kappa\in (0,8)$. From the argument, we see that, one approach to show that SLE$_\kappa$ admits an SLE$_\kappa(\rho)$ decomposition is to show that for any pre-compact measurable subset $U$ of $\HH$,  $\Psi^{\kappa,\rho}_t(U)$ defined in (\ref{Psi-U}) is of class $\cal D$, and then apply the Doob-Meyer decomposition theorem. But this is very difficult in general (cf.\ \cite{LS,LZ}). In the next section, we will show that SLE$_\kappa$ admits an SLE$_\kappa( -8)$ decomposition for any $\kappa>0$ using a different approach.

\section{Capacity Parametrization} \label{section-capacity}
Fix $\kappa>0$. We say that a chordal SLE$_\kappa$ curve $\gamma$ (with the reparametrization given by its definition) has capacity parametrization because the hull $K_t$ determined by $\gamma([0,t])$ has half-plane capacity $2t$ for each $t\ge 0$.
In this section we consider SLE$_\kappa(\rho)$ processes with $\rho=-8$, which turns out to be closely related to the capacity parametrization.

For any $z\in\HH$ and $N>0$, let $\PP^{\kappa,-8}_{z;N}$ be the measure on $(\Sigma,\F^0)$ defined by  $\PP^{\kappa,-8}_{z;N}(E)=\PP^{\kappa,-8}_z(E\sem \Sigma_N)$. Note that $\PP^{\kappa,-8}_{z;N}$ is not a probability measure. Then we have $\PP^{\kappa,-8}_{z;N}\ll\PP^{\kappa,-8}_z$, and so $\PP^{\kappa,-8}_{z;N}\tl\PP^{\kappa,-8}_z$. Since $\PP^{\kappa,-8}_z\tl\PP_\kappa$, we get $\PP^{\kappa,-8}_{z;N}\tl\PP_\kappa$.
We now calculate the local Radon-Nikodym derivatives of $\PP^{\kappa,-8}_{z;N}$ w.r.t.\ $\PP_\kappa$.

Let $\lambda_t=\sqrt\kappa B_t$, and $g_t$ the Loewner maps driven by $\lambda$.
Let $M_t(z)=M^{\kappa,-8}_t(z)$ and $G(z)=G^{\kappa,-8}(z)$ be as defined by (\ref{M-kappa})-(\ref{M-kappa2}). Then we have $G(z)=(\Imm z/|z|)^{8/\kappa}$; $M_t(z)=G(g_t(z)-\lambda_t)|g_t'(z)|^2$, $0\le t<T_z$; and $M_t(z)=0$, $T_z\le t<\infty$. 

Let $t\in[0,\infty)$ and $E\in\F_t^0\cap\Sigma_t$. If $t\ge N$, then $\PP^{\kappa,-8}_{z;N}(E)=0$ because $E\subset\Sigma_t\subset\Sigma_N$. If $t<N$, from (\ref{RN-rho}), we get
$$\PP^{\kappa,-8}_{z;N}(E)=\PP^{\kappa,-8}_z(E\sem \Sigma_N)=\PP^{\kappa,-8}_z(E)-\PP^{\kappa,-8}_z(E\cap \Sigma_N)$$
$$=\int_E \frac{M_t(z)}{G(z)} d\PP_\kappa-\int_{E\cap \Sigma_N} \frac{M_N(z)}{G(z)}d\PP_\kappa=\frac 1{G(z)}\int_E( M_t(z)-\EE_\kappa[ M_N(z)|\F_t^0])d\PP_\kappa.$$
Here the third ``='' holds because $E\in\F_t^0\cap\Sigma_t$ and $E\cap\Sigma_N\in\F_N^0\cap\Sigma_N$, and the fourth ``='' holds because $\PP_\kappa$ is supported by $\Sigma_\infty\subset\Sigma_N$. Thus,
\BGE \frac{d\PP^{\kappa,-8}_{z;N}|_{\F_t^0\cap\Sigma_t}}{d\PP_\kappa|_{\F_t^0\cap\Sigma_t}}=\left\{
\begin{array}{ll} \frac 1{G(z)}(M_{t }(z)-\EE_\kappa[ M_N(z)|\F_t^0]),& 0\le t<N;\\ 0,&N\le t<\infty.
\end{array}\right.\label{PNRN}\EDE

For $t\ge 0$, define $G_t$ on $\HH$ by
$$G_t(z)=M_0(z)-\EE_\kappa[M_t(z)]=G(z)-\EE_\kappa[M_t(z)],$$\
and let $C_{\kappa,t}=\int_\HH G_t(z)dA(z)$.
From (\ref{RN-rho}), we get
\BGE\frac{G_t(z)}{G(z)}=1-\int_{\Sigma_t} \frac{M_t(z)}{G(z)}d\PP_\kappa=1-\PP^{\kappa,\rho}_z[\Sigma_t]=\PP^{\kappa,\rho}_z[T_z\le t].\label{Tz<}\EDE
Thus, $G_0\equiv 0$, $G_t\ge 0$ for $t>0$, and
\BGE \lim_{t\to\infty} G_t(z)=G(z)\PP^{\kappa,\rho}_z[T_z<\infty]=G(z).\label{limGt}\EDE
From the scaling property of SLE$_\kappa(-8)$ process, we get
$\frac{G_{a^2 t}(az)}{G(az)}=\frac{G_t(z)}{G(z)}$ for any $a>0$.
Since $G(az)=G(z)$, we get $G_{a^2 t}(az)=G_t(z)$. Especially, we have
\BGE G_t(z)=G_1(\frac z{\sqrt t}),\quad t>0.\label{scalingGt}\EDE

\begin{Lemma}
  We have $C_{\kappa,t}=tC_{\kappa,1}$ for every $t\ge0$, and $C_{\kappa,1}\in(0,\infty)$. \label{LemC}
\end{Lemma}
\begin{proof}
 From (\ref{scalingGt}) and that $G_0\equiv 0$, we get $C_{\kappa,t}=tC_{\kappa,1}$, $t\ge 0$. From (\ref{limGt}) and (\ref{scalingGt}), we get $C_{\kappa,1}>0$.
Now we show $C_{\kappa,1}<\infty$. Fix $z_0=x_0+iy_0\in\HH$ and consider an SLE$_\kappa(-8)$ process started from $0$ with force point $z_0$.  We will use some results in Appendix \ref{section-bdd}. We may express $T_{z_0}$ by (\ref{Tz0}), where $V_s$ is a diffusion process that satisfies the SDE (\ref{dL}) with initial value $V_0=\arcsinh(x_0/y_0)$. So we immediately have $T_{z_0}\ge y_0^2\int_0^\infty e^{-4s}ds=y_0^2/4$. Thus, $T_{z_0}>1$ if $y_0>2$. From (\ref{Tz<}) we see that $G_1\equiv 0$ on $\{\Imm z>2\}$. From the left-right symmetry, we see that $G_1(x+iy)=G_1(-x+iy)$. We also know that $|G_1(z)|\le |G(z)|\le 1$ for all $z\in\HH$. Thus, to show that $C_{\kappa,1}=\int_\HH G_1(z) dA(z)<\infty$, it suffices to prove that there is $M>0$ such that $\int_M^\infty \int_0^2 G_1(x+iy) dydx<\infty$.

Suppose $|\sqrt\kappa B_s|\le 1+bs$ for every $s\ge 0$, where $b>0$ is to be determined. From (\ref{dL}) and that $|\tanh(x)|\le 1$, we have
$V_s\ge V_0-1-bs-(\frac\kappa 2+4)s$, $s\ge 0$. Thus,
$$\cosh^2(V_2)\ge \frac 14 e^{2V_s}\ge \frac 14 e^{2V_0}\cdot e^{-(2b+\kappa +8)s-2}\ge e^{-(2b+ \kappa+8)s-2} \sinh^2(V_0)=\frac{x_0^2}{y_0^2} e^{-(2b+ \kappa+8)s-2}.$$
From (\ref{Tz0}) we then get $T_{z_0}\ge \frac{e^{-2}x_0^2}{\kappa+12+2b}$, which then implies $T_{z_0}>1$ if $2b<e^{-2}x_0^2-\kappa-12$. From (\ref{Btab}) we know that $\PP[\sqrt\kappa | B_s|\le  1+bs,\forall s\ge 0]\ge 1-2e^{-\frac 2\kappa  b}$. Thus, by taking $2b=e^{-2}x_0^2-\kappa-12-\eps$ and letting $\eps\to 0^+$, we get
$$\PP^{\kappa,-8}_{z_0}[T_{z_0}\le 1]\le e^{-\frac 1\kappa (e^{-2}x_0^2-\kappa-12)}, \quad \mbox{if }e^{-2}x_0^2>\kappa+12.$$
Let $M=\sqrt{e^2(\kappa+12)}$. From (\ref{Tz<}) we get
$$G_1(z_0)\le G(z_0) \PP^{\kappa,-8}_{z_0}[T_{z_0}\le 1]\le \PP^{\kappa,-8}_{z_0}[T_{z_0}\le 1]\le e^{-\frac 1\kappa (e^{-2}x_0^2-\kappa-12)},\quad x_0>M.$$
So we get $\int_M^\infty\int_0^2  G_1(z_0)dydx<\infty$, as desired.
\end{proof}

Let $\mA $ be the Lebesgue measure on $\R$. For any measurable set $U\subset \lin\HH$, define
$$\Theta^{\kappa,-8}_t(U)=C_{\kappa,1}\mA(\gamma^{-1}(U)\cap [0,t]),\quad t\ge 0.$$
Especially, we have $\Theta^{\kappa,-8}_t:=\Theta^{\kappa,-8}_t(\lin\HH)=C_{\kappa,1} t$. We can now state the following theorem, which is similar to Theorem \ref{Thm1}.

\begin{Theorem} For any measurable set $U\subset \lin \HH$,
\BGE {\PP^{\kappa, -8}_{U}\ha\oplus\PP_\kappa}=\PP_\kappa \otimes d\Theta^{\kappa,-8}_\cdot(U),\label{pre-main-thm2}\EDE
  \BGE \Lo_*(\PP^{\kappa, -8}_z \oplus\PP_\kappa)\overleftarrow\otimes{\bf 1}_U G^{\kappa, -8}(z)dA(z)=\Lo_*(\PP_\kappa)\otimes {\cal M}^{\kappa,-8}_U,\label{int-U2}\EDE
where ${\cal M}^{\kappa,-8}_U$ is the kernel from $\Lo(\Sigma^\Lo)$ to $\lin\HH$ defined by ${\cal M}^{\kappa,-8}_U(\gamma,\cdot)=\gamma_*(d\Theta^{\kappa,-8}_\cdot(U) )$.
\label{Thm2}\end{Theorem}
\begin{proof}
  From the renewal property of Loewner's equation and the Markov property of $\sqrt\kappa B_t$, we see that, if $t\le N$, then
 $M_t(z)-\EE_\kappa[ M_N(z)|\F_t^0]={\bf 1}_{T_z>t} |g_t'(z)|^2 G_{N-t}(g_t(z)-\lambda_t)$.
Thus,
$$\int_\HH (M_t(z)-\EE_\kappa[ M_N(z)|\F_t^0])dA(z)=\int_{\HH\sem K_t} G_{N-t}(g_t(z)-\lambda_t) |g_t'(z)|^2 dA(z)$$
$$=\int_{\HH} G_{N-t}(w) dA(w)=C_{\kappa,N-t}=(N-t)C_{\kappa,1}.$$
Here in the second ``='' we use $w=g_t(z)-\lambda_t$ and the fact that $g_t$ maps $\HH\sem K_t$ conformally onto $\HH$.
The above formula together with Proposition \ref{Prop-fubini} and (\ref{PNRN}) implies that $\int_\HH \PP^{\kappa,-8}_{z;N} G(z) dA(z)\tl \PP_\kappa$, and the local Radon-Nikodym derivative is $C_{\kappa,1} (N-t)\vee 0$.  From Proposition \ref{Prop-trunc} we get
$$\int_\HH \PP^{\kappa,-8}_{z;N} G(z)dA(z)=\K_{C_{\kappa,1} d(t\wedge N)}(\PP_\kappa).$$
Letting $N\to \infty$, we get
\BGE \PP^{\kappa,-8}_\HH=\K_{C_{\kappa,1} dt}(\PP_\kappa)\label{killing-cap}.\EDE
Using Proposition \ref{Prop-Pkdecmp}, we get
$${\PP^{\kappa, -8}_{\HH}\ha\oplus\PP_\kappa}=\PP_\kappa \otimes C_{\kappa,1}\mA|_{[0,\infty)}.$$
Applying the extended Loewner map $\ha \Lo(\lambda,t)=(\Lo(\lambda),\Lo(\lambda)(t))$, we get
\BGE \Lo_*(\PP^{\kappa, -8}_z \oplus\PP_\kappa)\overleftarrow\otimes G(z)dA(z)=\Lo_*(\PP_\kappa)\otimes {\cal M}^{\kappa,-8},\label{int-U2'}\EDE
where ${\cal M}^{\kappa,-8}$ is the kernel from $\Lo(\Sigma^\Lo)$ to $\lin\HH$ defined by ${\cal M}^{\kappa,-8}_U(\gamma,\cdot)=\gamma_*(C_{\kappa,1}\mA|_{[0,\infty)} )$.  Note that ${\cal M}^{\kappa,-8}_U(\gamma,\cdot)$ is the restriction of ${\cal M}^{\kappa,-8} (\gamma,\cdot)$ to $U$. So we get (\ref{int-U2}) by restricting both sides of  (\ref{int-U2'}) to $\Sigma^\C\times U$. Finally, we get  (\ref{pre-main-thm2}) by applying $\ha\Lo^{-1}$.
\end{proof}

The following two corollaries are similar to Corollaries \ref{Cor1} and \ref{Cor2}.

\begin{Corollary}
   Let $U$ be a measurable subset of $\HH$ with $\int_U G^{\kappa,-8}(z)dA(z)<\infty$. If we integrate the laws of extended  SLE$_\kappa(-8)$ curve started from $0$ with force point at $z$ against the measure ${\bf 1}_U G^{\kappa,-8}(z)dA(z)$, then we get a finite measure on curves, which is absolutely continuous w.r.t.\ the law of an SLE$_\kappa$ curve, and the Radon-Nikodym derivative is $|{\cal M}^{\kappa,-8}_U(\gamma,\cdot)|=C_{\kappa,1}\mA(\gamma^{-1}(U))$. \label{Cor3}
\end{Corollary}

\begin{Corollary}
  Suppose $( \gamma, t)$ is a $C([0,\infty),\C)\times [0,\infty)$-valued random variable with the properties that $ \gamma$ has the law of an SLE$_\kappa$ curve, and given $ \gamma$, the law of  $t$ is absolutely continuous w.r.t.\  the Lebesgue measure on $[0,\infty)$. Then the law of $\gamma(t)$ is absolutely continuous w.r.t.\ ${\bf 1}_\HH dA(z)$, and the law of $ \gamma$ conditioned on $z=\gamma(t)$ is absolutely continuous w.r.t.\ the law of an extended  SLE$_\kappa(-8)$ curve started from $0$ with force point at $z$. \label{Cor4}
\end{Corollary}

\no{\bf Remark.} From Corollary \ref{Cor4} we see that, if we sample a point on an SLE$_\kappa$ curve according to a law that is absolutely continuous w.r.t.\ the capacity parametrization, and stop the SLE$_\kappa$ curve at that point, then the law of the stopped curve conditioned on that point is absolutely continuous w.r.t.\ the law of SLE$_\kappa(-8)$ curve. This extends a result in \cite{tip}, whose argument used the symmetry of backward SLE$_\kappa$ welding for $\kappa\in(0,4]$ (cf.\ \cite{RZ}) and the conformal removability of SLE$_\kappa$ curves for $\kappa\in(0,4)$ (cf.\ \cite{JS,RS}).

    More specifically, from \cite[Remark 2 after Theorem 6.6]{tip}, we know that, if $\kappa\in(0,4)$, then the above conditional stopped curve is the conformal image of an initial segment of a whole-plane SLE$_\kappa(\kappa+2)$ curve, which is also an end segment of a whole-plane SLE$_\kappa(\kappa+2)$ curve, thanks to the reversibility of whole-plane SLE$_\kappa(\rho)$ curve (cf.\ \cite{MS4}). From Proposition \ref{coordinate}, we know that an end segment of a whole-plane SLE$_\kappa(\kappa+2)$ curve can be mapped conformally to an end segment of a chordal SLE$_\kappa(-8)$ curve. Thus, Corollary \ref{Cor4} extends a weaker version of \cite[Remark 2 after Theorem 6.6]{tip} from $\kappa\in(0,4)$ to $\kappa\in(0,\infty)$. The result here is weaker because we can not conclude that the conditional stopped curve is exactly the conformal image of an SLE$_\kappa(-8)$ curve, but can only say that its law is {\it absolutely continuous} w.r.t.\ the law of an SLE$_\kappa(-8)$ curve.

\begin{Corollary}
  Let $\gamma$ be an SLE$_\kappa$ curve. Then for any measurable set $U\subset \lin\HH$, we have  $\EE[\mA(\gamma^{-1}(U))]=\frac 1{C_{\kappa,1}}\int_U G^{\kappa,-8}(z) dA(z)$. Especially, we have a.s.\ $\mA(\gamma^{-1}(\R))=0$. \label{average}
\end{Corollary}
\begin{proof}
  The first statement follows from computing the total mass of the measures in (\ref{int-U2})  and the fact that $|{\cal M}^{\kappa,-8}_U|=C_{\kappa,1}\mA(\gamma^{-1}(U))$. The second statement follows from taking $U=\R$ and the fact that $\int_R G^{\kappa,-8}(z) dA(z)=0$.
\end{proof}

\no{\bf Remark.} The above corollary says that $\frac 1{C_{\kappa,1}} G^{\kappa,-8}(z)$ is the density of SLE$_\kappa$ curve in capacity parametrization, and so may be called the capacity Green's function for SLE$_\kappa$.

\begin{Corollary}
  Let $\gamma$ be an SLE$_\kappa$ curve. Then there is a random conformal map $W$ from $\HH$ into $\ha\C$ such that $W(\gamma(1))=\infty$, and the law of $W(\gamma(t))$, $0\le t\le 1$, is absolutely continuous w.r.t.\ an {\it end} segment of a whole-plane SLE$_\kappa(\kappa+2)$ curve, up to a reparametrization. \label{Cor5}
\end{Corollary}
\begin{proof}
  Let $\til\gamma$ be an SLE$_\kappa$ curve independent of $\gamma$. Let $\til\lambda_t$ and $\til g_t$ be the driving function and Loewner maps associated with $\til\gamma$. Let $\xi$ be a random variable with law ${\bf 1}_{[0,1]}dx$ that is independent of $\til\gamma$ and $\gamma$. Define $\ha\gamma$ such that $\ha\gamma(t)=\til\gamma(t)$ for $0\le t\le \xi$ and $\ha\gamma(t)=\til f_\xi(\til\lambda_\xi+\gamma(t-\xi))$ for $\xi\le t<\infty$, where $\til f_\xi$ is the continuation of $\til g_\xi^{-1}$ to $\lin\HH$. From the domain Markov property of SLE$_\kappa$, we see that $\ha\gamma$ is an SLE$_\kappa$ curve independent of $\xi$. Since $\xi+1$ has the law ${\bf 1}_{[1,2]}dx$, from Corollary \ref{Cor4}, we see that, conditioned on $z_0=\ha\gamma(\xi+1)$, the law of $\ha\gamma_t$, $0\le t\le \xi+1$, is absolutely continuous w.r.t.\ the law of an SLE$_\kappa(-8)$ curve started from $0$ with force point at $z_0$. Let $W_1$ be a conformal map from $\HH$ onto $\D:=\{|z|<1\}$ such that $W_1(0)=1$ and $W_1(z_0)=0$. From Proposition \ref{coordinate}, we know that $W_1$ maps an SLE$_\kappa(-8)$ curve started from $0$ with force point at $z_0$ to a radial SLE$_\kappa(\kappa+2)$ curve in $\D$ started from $1$ with force point at $W_1(\infty)$, up to a reparametrization.

  By \cite{Law1}, there is a conformal map $W_2$ from $\D$ into $\ha\C$ such that $W_2(0)=\infty$, and $W_2$ maps the radial SLE$_\kappa(\kappa+2)$ curve to an end segment of a whole-plane SLE$_\kappa(\kappa+2)$ curve. Let $W_0(z)=\til f_\xi(\til\lambda_\xi+z)$, which is a conformal map from $\HH$ into $\HH$. Then $W:=W_2\circ W_1\circ W_0$ is the conformal map we are looking for.
\end{proof}

\no{\bf Remark.} Corollary \ref{Cor5} extends a weaker version of \cite[Theorem 5.3]{tip}, which states that, for $\kappa\in(0,4)$, there is a random conformal map $W$ from $\HH$ into $\ha\C$ such that $W(\gamma(1))=0$, and  $W(\gamma(t))$, $0\le t\le 1$, is  an {\it initial} segment of a whole-plane SLE$_\kappa(\kappa+2)$ curve, up to a reparametrization.

\begin{Corollary}  If $\int_U G^{\kappa,-8}(z)dA(z)<\infty$, then $M^{\kappa,-8}_t(U):=\Psi^{\kappa,-8}_t(U)+\Theta^{\kappa,-8}_t(U)$ is a uniformly integrable $(\F^0_t)$-martingale.
\end{Corollary}
\begin{proof} It suffices to show that $M^{\kappa,-8}_t(U)=\EE_\kappa[C_{\kappa,1}\mA(\gamma^{-1}(U))|\F^0_t]$ for every $t\ge 0$. We have $ C_{\kappa,1}\mA(\gamma^{-1}(U))= C_{\kappa,1}\mA(\gamma^{-1}(U)\cap [0,t])+C_{\kappa,1}\mA(\gamma^{-1}(U)\cap [t,\infty])$, and $C_{\kappa,1}\mA(\gamma^{-1}(U)\cap [0,t])=\Theta^{\kappa,-8}_t(U)$ is $\F^0_t$-measurable. Thus, it remains to show that
$$\EE_\kappa [C_{\kappa,1}\mA(\gamma^{-1}(U)\cap [t_0,\infty])|\F^0_{t_0}]=\Psi^{\kappa,-8}_{t_0}(U),\quad t_0\ge 0.$$
Fix $t_0\ge 0$. By the domain Markov Property of SLE$_\kappa$, there is an SLE$_\kappa$ curve $\ha\gamma$ independent of $\F^0_{t_0}$ such that $\gamma(t_0+t)= f_{t_0}(\lambda_{t_0}+\ha\gamma(t))$ for all $t\ge 0$, where $f_{t_0}$ is the continuation of $g_{t_0}^{-1}$ to $\lin\HH$, and $\lambda_t$ and $g_t$ are the driving function and Loewner maps associated with $\gamma$. Let $\til g_{t_0}(z)=g_{t_0}(z)-\lambda_{t_0}$ and $\til f_{t_0}(z)=f_{t_0}(\lambda_{t_0}+z)$. Conditioned on $\F^0_{t_0}$, we have
$$\mA(\gamma^{-1}(U)\cap [t_0,\infty])=\mA(\ha\gamma^{-1}(\til f_{t_0}^{-1}(U)))=\mA(\ha\gamma^{-1}(\til f_{t_0}^{-1}(U)\cap\HH))$$
$$=\mA(\ha\gamma^{-1}(\til g_{t_0}(U\cap(\HH\sem K_{t_0}))) $$
From Corollary \ref{average}, we get
$$\EE_\kappa[ \mA(\gamma^{-1}(U)\cap [t_0,\infty])|\F^0_{t_0}]=\int_{\til g_{t_0}(U\cap(\HH\sem K_{t_0}))} G^{\kappa,-8}(z)dA(z)$$
$$=\int_{U\cap(\HH\sem K_{t_0})} G^{\kappa,-8}( g_{t_0}(z)-\lambda_{t_0}) |g_{t_0}'(z)|^2dA(z) =\int_{U} M^{\kappa,-8}_{t_0}(z)dA(z)=\Psi^{\kappa,-8}_{t_0}(U).$$
This finishes the proof.
\end{proof}

\section{Intersection of SLE with the boundary} \label{section-boundary}
In this section, we decompose SLE$_\kappa$ into  SLE$_\kappa(\rho)$ processes with the force point lying on $\R$.   An SLE$_\kappa(\rho)$ process started from $a_0\in\R$ with force point at $x_0\in\R\sem\{a_0\}$ is the solution of the Loewner equation driven by $\lambda_t$, $0\le t<T_{x_0}$, which is the solution of the SDE:
$$ d\lambda_t=\sqrt\kappa dB_t+ \frac{\rho}{\lambda_t-g^{\lambda }_t(x_0)},\quad \lambda_0=a_0. $$
The Loewner curve driven by $\lambda$, which a.s.\ exists, is called an SLE$_\kappa(\rho)$ curve started from $a_0$ with force point at $x_0$.

Let $\PP^{\kappa,\rho}_{x_0}$ denote the law of the driving SLE$_\kappa(\rho)$ process started from $0$ with force point at $x_0\in\R\sem\{0\}$. From Proposition \ref{Prop-Girs}, we know that $\PP^{\kappa,\rho}_{x_0}\tl\PP_\kappa$. We now derive the local Radon-Nikodym derivative. Let $\lambda_t=\sqrt\kappa B_t$ and $g_t$ be the Loewner maps driven by $\lambda$.
Define
$$ M^{\kappa,\rho}_t(x_0)=|g_t(x_0)-\lambda_t|^{\frac\rho\kappa}\cdot g_t'(x_0)^{\frac\rho\kappa(1-\frac \kappa 4+\frac\rho 4)},\quad 0\le t<T_{x_0},$$ 
and $$ G^{\kappa,\rho}(x_0)=M^{\kappa,\rho}_0(x_0)=|x_0|^{\frac\rho\kappa} ,\quad z\in\HH.$$ 
Direct calculation using It\^o's formula shows that $(M^{\kappa,\rho}_t(x_0))$ is an $(\F^0_t)$-adapted continuous local martingale, and satisfies the SDE:
$$ dM^{\kappa,\rho}_t(x_0)=M^{\kappa,\rho}_t(x_0)\cdot   \frac{\rho/\sqrt\kappa}{\lambda_t-g_t(x_0)}\,dB_t,\quad 0\le t<T_{x_0}.$$ 
We further define $M^{\kappa,\rho}_t(x_0)=0$ for $t\in[ T_{x_0},\infty)$. From Proposition \ref{Prop-Girs} we know that $$ \frac{d\PP^{\kappa,\rho}_{x_0}|_{\F^0_t\cap\Sigma_t}}{d\PP_\kappa|_{\F^0_t\cap\Sigma_t}}= \frac{M^{\kappa,\rho}_t(x_0)}{G^{\kappa,\rho} (x_0)},\quad 0\le t<\infty.$$ 
For a measurable subset $U$ of $\R$, define
$$ \Psi^{\kappa,\rho}_t(U)=\int_{U }  M^{\kappa,\rho}_t(x)dx,\quad \PP^{\kappa,\rho}_U=\int_U  \PP^{\kappa,\rho}_x G^{\kappa,\rho} (x) dx.$$ 
For each $x\in\R\sem\{0\}$, $(M^{\kappa,\rho}_t(x))$ is a supermartingale because it is a positive local martingale. Thus, if $\int_U G^{\kappa,\rho}(x)dx<\infty$, then $(\Psi^{\kappa,\rho}_t(U))$ is also a supermartingale, which has to be $\PP_\kappa$-a.s.\ finite.
From Proposition \ref{Prop-fubini}, we have
$$ \frac{d\PP^{\kappa,\rho}_{U}|_{\F^0_t\cap\Sigma_t}}{d\PP_\kappa|_{\F^0_t\cap\Sigma_t}}= \Psi^{\kappa,\rho}_t(U),\quad 0\le t<\infty. $$ 

It is known that (cf.\ \cite{duality}), if $\rho\le \frac\kappa 2-4$, then for an SLE$_\kappa(\rho)$ curve $\gamma$ with force point $x_0$, we have a.s.\ $\lim_{t\to T_{x_0}^-} \gamma(t)=x_0$, which then implies that $T_{x_0}<\infty$, and $\lim_{t\to T_{x_0}^-}\lambda_t$ converges, where $\lambda$ is the corresponding driving function. This means that $\PP^{\kappa,\rho}_{x_0}$ is supported by $\Sigma^\oplus$, and we may define $\PP^{\kappa,\rho}_{x_0}\oplus \PP_\kappa$. The pushforward measure $\Lo_*(\PP^{\kappa,\rho}_{x_0}\oplus \PP_\kappa)$ is called the law of an extended SLE$_\kappa(\rho)$ curve started from $0$ with force point at $x_0$.

For $\kappa\in(0,8)$ and $\rho=\kappa-8$, an extended SLE$_\kappa(\rho)$ curve with force point at $x_0$ is also called a two-sided chordal SLE$_\kappa$ curve through $x_0$. It can be understood as an SLE$_\kappa$ curve conditioned to pass though $x_0$. To make this rigorous, one may condition an SLE$_\kappa$ curve $\gamma$ on the event that $\dist(x_0,\gamma)<r$, and then pass the limit $r\to 0$. A two-sided chordal SLE$_\kappa$ curve also satisfies reversibility, which is similar to that of a two-sided radial SLE$_\kappa$ curve.

It is proved in \cite{AlbertsSheffield} that, for $\kappa\in(4,8)$, if $U$ is a pre-compact measurable subset of $\R\sem\{0\}$, then $(\Psi^{\kappa,\kappa-8}_t(U))$ is of class $\cal D$, and Doob-Meyer decomposition theorem can be applied to get a unique continuous increasing function $\Theta_t(U)$ such that $\Theta_0(U)=0$ and $\Psi_t(U)+\Theta_t(U)$ is a uniformly integrable martingale. We may then define $\Theta_t(U)$ for any measurable subset $U$ of $\R$ using a limiting procedure. It is conjectured (cf.\ \cite{AlbertsSheffield} and \cite{Law4}) that $\Theta_t(U)$ agrees up to a multiplicative constant with the $d'$-dimensional Minkowski content of $\gamma([0,t])\cap U$, where $d':=2-\frac 8\kappa$ is the Hausdorff dimension of $\gamma\cap\R$ (cf.\ \cite{AlbertsSheffield-H}). Using the argument in Section \ref{section-NP}, we can obtain the following theorem and corollaries.

\begin{Theorem} Let $\kappa\in(4,8)$. Let $U$ be any measurable subset of $\R$. Then we have
$${\PP^{\kappa,\kappa-8}_U\ha\oplus\PP_\kappa}=\PP_\kappa\otimes d\Theta(U),$$ 
 $$ \Lo_*(\PP^{\kappa,\kappa-8}_x \oplus\PP_\kappa)\overleftarrow\otimes{\bf 1}_U G^{\kappa,\kappa-8}(x)dx=\Lo_*(\PP_\kappa)\otimes {\cal M}_U,$$ 
  where ${\cal M}_U$ is the $\Lo_*(\PP_\kappa)$-kernel from $\Lo(\Sigma^\Lo)$ to $\lin\HH$ defined by ${\cal M}_U(\gamma,\cdot)=\gamma_*(d\Theta(U))$. 
\label{Thm3}\end{Theorem}

\begin{Corollary}
  Let $\kappa\in(4,8)$. Let $U$ be a measurable subset of $\R$ with $\int_U G^{\kappa,\kappa-8}(x)dx<\infty$. If we integrate the laws of two-sided chordal SLE$_\kappa$ through $x$ against the measure ${\bf 1}_U G^{\kappa,\kappa-8}(x)dx$, then we get a finite measure on curves, which is absolutely continuous w.r.t.\ the law of an SLE$_\kappa$ curve, and the Radon-Nikodym derivative is $|{\cal M}_U(\gamma,\cdot)|=\Theta_\infty(U)(\Lo^{-1}(\gamma))$. \label{Cor8}
\end{Corollary}

\begin{Corollary}
  Let $\kappa\in(4,8)$. Suppose $( \gamma, x)$ is a $C([0,\infty),\C)\times\R$-valued random variable with the properties that $ \gamma$ has the law of an  SLE$_\kappa$ curve, and given $ \gamma$, the law of  $x$ is absolutely continuous w.r.t.\  ${\cal M}_\R(\gamma,\cdot)$. Then the law of $x$ is absolutely continuous w.r.t.\ $\mA$, and the law of $ \gamma$ given $x$ is absolutely continuous w.r.t.\ the law of a two-sided chordal SLE$_\kappa$ curve through $x$. \label{Cor9}
\end{Corollary}

\no{\bf Remark.} Since an SLE$_\kappa$ curve for $\kappa\in(0,4]$ does not visit any point on $\R\sem\{0\}$, Corollary \ref{Cor8} does not hold in the case $\kappa\in(0,4]$, and neither does the main result in \cite{AlbertsSheffield}.

\section{Decomposition of Planar Brownian motion} \label{Section-BM}
We now use the argument in the proof of Theorem \ref{Thm1} to decompose a planar Brownian motion. We modify the definition of $\Sigma$ such that the continuous functions take values in $\R^2$. Let $D\subset\R^2$ be a bounded domain that contains $0=(0,0)$. Let $B_t$ be a planar Brownian motion that starts from $0$, and $\tau_D$ the first time that $B_t$ exists $D$. Let $G_D(\cdot,\cdot)$ denote the  Dirichlet Green's function in $D$ for Brownian motion. Fix $z_0=(x_0,y_0)\in D\sem\{0\}$. Consider the Doob's $h$-transform of $B_t$ with $h=G_D(\cdot,z_0)$. This a diffusion process that satisfies the SDE:
$$dZ^{z_0}_t=dB_t+\frac{\nabla G_D(\cdot,z_0)|_{Z^{z_0}_t}}{G_D(Z_t^{z_0},z_0)}\,dt,\quad Z^{z_0}_0=0.$$
We may view $Z^{z_0}_t$ as $B_t$ conditioned to visit $z_0$ before exiting $D$.
Let $\PP$ and $\PP^D_{z_0}$ denote the laws of $(B_t)$ and $(Z^{z_0}_t)$, respectively. Then $\PP^D_{z_0}\tl \PP$, and
$$\frac{d\PP^D_{z_0}|_{\F^0_t\cap \Sigma_t}}{d\PP|_{\F^0_t\cap\Sigma_t}}=\frac{G_D(B_{t\wedge \tau_D},z_0)}{G_D(0,z_0)},\quad 0\le t<\infty.$$
Let $f(w)=\int_D G_D(w,z)dA(z)$. Then $f$ is the solution of the Dirichlet problem $\Delta f\equiv -1$ in $D$ and $f\equiv 0$ on $\pa D$. Define
$$\PP^D=\int_D \PP^D_{z}G_D(0,z)dA(z),\quad \Psi_t=\int_D G_D(B_{t\wedge \tau_D},z)dA(z)=f(B_{t\wedge \tau_D}).$$
From Proposition \ref{Prop-fubini}, we see that $\PP^D\tl \PP$ and
\BGE \frac{d\PP^D |_{\F^0_t\cap \Sigma_t}}{d\PP|_{\F^0_t\cap\Sigma_t}}=\Psi_t,\quad 0\le t<\infty.\label{PYPD}\EDE
From It\^o's formula and the PDE for $f$, we see that $M_t:=\Psi_t+\frac t2$, $0\le t<\tau_D$, is a continuous local martingale, and $\lim_{t\to\tau_D^-} M_t=\frac12{\tau_D}$. We further define $M_t=\frac12{\tau_D}$ for $t\ge \tau_D$. Since $D$ is bounded, the probability $\PP[\tau_D>N]$ decays exponentially as $N\to\infty$. Also note that $f$ is bounded. Thus, $(M_t)$ is a uniformly integrable martingale, and we have $$\EE[\frac{\tau_D}2|\F^0_t]=\EE[M_\infty|\F^0_t]=M_t=\Psi_t+\frac {t\wedge \tau_D}2,\quad 0\le t<\infty.$$
Define a process $(\theta_t)$ such that $\theta_t=\frac {t\wedge \tau_D}2$. Then we have
\BGE \EE[\theta_\infty|\F^0_t]-\theta_t=\Psi_t,\quad 0\le t<\infty. \label{theta-theta}\EDE
From Proposition \ref{Prop-trunc}, (\ref{PYPD}) and (\ref{theta-theta}), we see that
\BGE \PP^D=\K_{d\theta}(\PP).\label{end}\EDE
Apply first the operation $\ha\oplus\PP$ and then the pushforward by the map $(B,t)\mapsto (B,B(t))$ to both sides of (\ref{end}). Using a variation of Proposition \ref{Prop-Pkdecmp}, we obtain the following theorem.

\begin{Theorem}
  Let $\mA$ and $dA$ denote the Lebesgue measures on $\R$ and $\R^2$, respectively. Then we have
  $$ (\PP^D_z\oplus \PP)\overleftarrow\otimes  G_D(z,0)dA(z)=\PP\otimes\cal M,$$ 
  where $\cal M$ is a kernel from $\Sigma$ to $D$ defined by ${\cal M}(B,\cdot)=\frac 12 B_*(\mA|_{[0,\tau_D]})$.
\end{Theorem}

\appendixpage
\addappheadtotoc
\appendix

\section{Boundedness of  SLE$_\kappa(\rho)$}\label{section-bdd}
In this section, we prove Proposition \ref{transience-0} (i). Fix $\kappa>0$ and $\rho\le \frac\kappa 2-4$. Fix $z_0=x_0+iy_0\in\HH$. Let $\gamma$ be an SLE$_\kappa(\rho)$ curve with force point at $z_0$. Let $\lambda(t)$, $0\le t<T_{z_0}$, be the driving function, and $g_t$ be the corresponding Loewner maps. Define $X_t$, $Y_t$, $D_t$ using (\ref{ZXYD}). Then the ODEs (\ref{XYD}) still hold, and now $X_t$ satisfies the SDE:
\BGE dX_t=-\sqrt\kappa dB_t+\frac{(\rho+2)X_t}{X_t^2+Y_t^2}dt.\label{X'}\EDE
Let $u(t)=\int_0^t \frac 1{X_s^2+Y_s^2}\,ds$. Let $\ha X_s=X_{u^{-1}(s)}$, $\ha Y_s=Y_{u^{-1}(s)}$, and $\ha R_s=\frac{\ha X_s}{\ha Y_s}$. Using (\ref{XYD}) and (\ref{X'}), we find that
$\ha Y_{s}=y_0 e^{-2s}$ and there is a standard Brownian motion $\ha B_s$ such that
\BGE d \ha R_s=\sqrt\kappa\sqrt{1+\ha R_s^2}d\ha B_s+(\rho+4) \ha R_sds.\label{haR}\EDE
Let $V_s=\arcsinh(\ha R_s)$. From It\^o's formula, we get
\BGE dV_s=\sqrt\kappa d\ha B_s+(\rho+4-\frac\kappa 2)\tanh(V_s)ds.\label{dL}\EDE
From $u(t)=\int_0^t \frac 1{X_s^2+Y_s^2}\,ds$, $\ha Y_{s}=y_0 e^{-2s}$, and $V_s=\arcsinh(\ha R_s)$, we get
\BGE T_{z_0}=\int_0^\infty \ha Y_{s}^2(1+\ha R_s^2)ds=y_0^2\int_0^\infty e^{-4s}\cosh^2(V_s)ds.\label{Tz0}\EDE

\begin{Lemma}
  Suppose $(V_s)$ solves (\ref{dL}) with $\rho\le \frac\kappa 2-4$. Then $(V_s)$ can be coupled with a standard Brownian motion $\til B_s$ such that $|V_s|\le |V_0+\sqrt\kappa\til B_s|$ for all $s$. \label{local}
\end{Lemma}
\begin{proof}
  From It\^o-Tanaka's formula, $|V_s|$ satisfies the  SDE:
  $$ d|V_s|= \sign(V_s)\sqrt\kappa d\ha B_s +(\rho+4-\frac\kappa 2)\tanh(|V_s|)ds+dL^V_s,$$
  where $L^V_\cdot$ is continuous and increasing, and stays constant on the intervals on which $|V_\cdot|>0$. On the other hand, if $\til B_s$ is a standard Brownian motion, then $|V_0+\sqrt\kappa\til B_s|$ satisfies the  SDE:
  $$d|V_0+\sqrt\kappa\til B_s|=\sign(V_0+\sqrt\kappa\til B_s)\sqrt\kappa d\til B_s+dL^B_s,$$
  where $L^B_\cdot$ is continuous and increasing. Define two standard Brownian motions $B^{(1)}_t$ and $B^{(2)}_t$ such that $B^{(1)}_t=\int_0^t \sign(V_s)  d\ha B_s$ and $B^{(2)}_t=\int_0^t \sign(V_0+\sqrt\kappa\til B_s)  d\til B_s$. We may couple $\til B_s$ with $\ha B_s$, and so with $V_s$, such that $B^{(1)}_\cdot=B^{(2)}_\cdot$. Since $\rho+4-\frac\kappa 2\le 0$, from the SDE for $|V_s|$ and $|V_0+\sqrt\kappa\til B_s|$, we find that $  |V_s|-\sqrt\kappa B^{(1)}_s-L^V_s$ is decreasing, and  $|V_0+\sqrt\kappa\til B_s|-\sqrt\kappa B^{(1)}_s$ is increasing.
  Suppose that there is $s_0\ge 0$ such that  $|V_{s_0}|> |V_0+\sqrt\kappa\til B_{s_0}|$. Then $s_0>0$. Let $s_0'\in(0,s_0)$ be such that $|V_{s_0'}|= |V_0+\sqrt\kappa\til B_{s_0'}|$ and $|V_{s}|> |V_0+\sqrt\kappa\til B_{s}|$ for $s\in(s_0',s_0]$. Then $|V_s|>0$ on $(s_0',s_0]$. So $L^V_s$ stays constant on $[s_0',s_0]$, which implies that $  |V_s|-\sqrt\kappa B^{(1)}_s$ is decreasing on $[s_0',s_0]$. Since $|V_0+\sqrt\kappa\til B_s|-\sqrt\kappa B^{(1)}_s$ is increasing, we conclude that $|V_{s}|- |V_0+\sqrt\kappa\til B_{s}|$ is decreasing on $[s_0',s_0]$, which contradicts that $|V_{s_0}|> |V_0+\sqrt\kappa\til B_{s_0}|$ and $|V_{s_0'}|= |V_0+\sqrt\kappa\til B_{s_0'}|$. Thus, $|V_{s}|\le |V_0+\sqrt\kappa\til B_{s}|$ for all $s\ge 0$.
\end{proof}

We will use the following well-known inequalities about Brownian motions:
\BGE \PP[|\sqrt\kappa B_t|\le at+b,\forall t\in[0,\infty)]\ge 1-2e^{-\frac{2ab}\kappa},\quad a,b>0.\label{Btab}\EDE

\begin{proof}[Proof of Proposition \ref{transience-0} (i).]
From (\ref{Tz0}) we get $T_{z_0} \le y_0^2\int_0^\infty e^{-4s}e^{2|V_s|}ds$.
From Lemma \ref{local}, $(V_s)$ may be coupled with a standard Brownian motion $\til B_s$ such that $|V_s|\le |V_0+\sqrt\kappa\til B_s|$.
Since $e^{2|V_0|}\le 4\cosh^2(V_0)=4(1+\ha R_0^2)=4\frac{|z_0|^2}{y_0^2}$, we get
$$ T_{z_0}\le  y_0^2\int_0^\infty e^{-4s}e^{2|V_0|+2\sqrt\kappa|\til B_s|}ds\le 4|z_0|^2\int_0^\infty e^{2\sqrt\kappa|\til B_s|-4s}ds.$$ 
From (\ref{Btab}) we see that the probability that $\sqrt\kappa |\til B_s|\le  s+b$ for all $s\ge 0$ is at least $1-2e^{-\frac 2\kappa  b}$, and on this event, from the above formula we get $T_{z_0}\le 2|z_0|^2e^{2b}$. Thus,
\BGE \PP^{\kappa,\rho}_{z_0}[T_{z_0}\le2|z_0|^2e^{2b}]\ge 1-2e^{-\frac 2\kappa b},\quad b>0.\label{T<}\EDE
This implies that $\PP^{\kappa,\rho}_{z_0}$-a.s., $T_{z_0}<\infty$.

Recall that $\lambda_t=\sqrt\kappa B_t+\int_0^t \frac{-\rho X_r}{X_r^2+Y_r^2}dr$, $0\le t<T_{z_0}$. The finiteness of $T_{z_0}$ implies that a.s.\ $\lim_{t\to T_{z_0}^-}B_t\in\R$. For the other term, consider
$$\int_0^{T_{z_0}} \frac{| X_t|}{X_t^2+Y_t^2}dt=\int_0^\infty |\ha X_s|ds=\int_0^\infty \ha Y_s |\sinh(V_s)|ds\le y_0\int_0^\infty e^{|V_s|-2s}ds.$$
Coupling $(V_s)$ with $(\til B_s)$ using Lemma \ref{local} and then using (\ref{Btab}), we find that $\PP^{\kappa,\rho}_{z_0}$-a.s., $\sup_{s\ge 0}(|V_s|-s)<\infty$. Thus, $\PP^{\kappa,\rho}_{z_0}$-a.s., $\int_0^{T_{z_0}} \frac{| X_t|}{X_t^2+Y_t^2}dt<\infty$, which implies that the limit $\lim_{t\to T_{z_0}^-} \int_0^t \frac{ X_r}{X_r^2+Y_r^2}dr$ exists and is finite. Thus, $\PP^{\kappa,\rho}_{z_0}$-a.s., $\lim_{t\to T_{z_0}^-}\lambda_t\in\R$.
\end{proof}

\begin{Corollary}
  Almost surely $\gamma([0,T_{z_0}))$ is bounded. \label{bound1}
\end{Corollary}
\begin{proof}
  This follows from Proposition \ref{transience-0} (i) and Lemma 4.1 in \cite{Law1}.
\end{proof}

\begin{Lemma}
  For any bounded measurable $U\subset\lin\HH$, $\PP_\kappa$-a.s., $\Psi^{\kappa,\rho}_t(U)\to 0$ as $t\to\infty$. \label{Psi-infty}
\end{Lemma}
\begin{proof}
  Taking  $e^{2b}=t/(2|z_0|^2) $ in (\ref{T<}) for some $t>0$, we get
$$ \PP^{\kappa,\rho}_{z_0}[T_{z_0}> t ]\le 2^{1+1/\kappa}  (t/|z_0|^2)^{-1/\kappa},\quad t>0. $$
Using (\ref{RN-rho}), we get
$$\EE_\kappa[M^{\kappa,\rho}_t(z_0)]= G^{\kappa,\rho} (z_0) \PP^{\kappa,\rho}_{z_0}[T_{z_0}> t ]\le 2^{1+\frac 1\kappa} |z_0|^{\frac{\rho+2}\kappa+\frac{\rho^2}{8\kappa}}\cdot t^{-\frac1{ \kappa}},\quad t>0.$$ 
Suppose $|z|\le R$ for every $z\in U$. Then
$$\EE_\kappa[\Psi^{\kappa,\rho}_t(U)]=\int_U \EE_\kappa[M^{\kappa,\rho}_t(z)]dA(z)\le \pi 2^{ \frac 1\kappa}R^{2+\frac{\rho+2}\kappa+\frac{\rho^2}{8\kappa}}\cdot t^{-\frac1{ \kappa}},\quad t>0.$$
Thus, $\lim_{t\to\infty} \EE_\kappa[\Psi^{\kappa,\rho}_t(U)]=0$. Since $(\Psi^{\kappa,\rho}_t(U))$ is a positive supermartingale, from Doob's martingale convergence theorem, we see that $\PP_\kappa$-a.s., $\lim_{t\to\infty} \Psi^{\kappa,\rho}_t(U)$ exists. From Fatou's lemma,
$\EE_\kappa [\lim_{t\to\infty} \Psi^{\kappa,\rho}_t(U)]\le \lim_{t\to\infty} \EE_\kappa[\Psi^{\kappa,\rho}_t(U)]=0$. So we get the conclusion.
\end{proof}

\no{\bf Remark.} Proposition \ref{transience-0} (i) and Lemma \ref{Psi-infty} also hold if $\frac\kappa 2-4<\rho<\frac\kappa 2-2$. In that case we may use the estimate  $|V_s|\le |V_0|+\sqrt\kappa |\ha B_s|+(\rho+4-\frac\kappa 2)s$, which follows from (\ref{dL}) and that $|\tanh(x)|\le 1$ for $x\in\R$. Then we may apply (\ref{Btab}) with $a\in(0, \frac\kappa 2-2-\rho)$. 

\end{document}